\documentclass{amsart}
\usepackage[dvips]{graphics}
\newtheorem{theorem}{Theorem}[section]

\theoremstyle{definition}
\newtheorem{definition}[theorem]{Definition}
\newtheorem{example}[theorem]{Example}

\theoremstyle{remark}
\newtheorem{remark}[theorem]{Remark}

\numberwithin{equation}{section}

\theoremstyle{problem}
\newtheorem{problem}[theorem]{Problem}

\newcommand{\sgn}{\mathop{\rm sgn}\nolimits}

\newcommand{\I}{\mathop{\rm Im}\nolimits}
\newcommand{\R}{\mathop{\rm Re}\nolimits}
\newcommand{\const}{\mbox{const}}

\newcommand{\Md}{\partial}

\newcommand{\ov}[1]{\overline{#1}}

\newcommand{\Ga}{\alpha}
\newcommand{\Gb}{\beta}
\newcommand{\Gd}{\delta}
\newcommand{\Ge}{\epsilon}

\newcommand{\Gf}{\phi}
\newcommand{\Gvf}{\varphi}
\newcommand{\Gg}{\gamma}
\newcommand{\Gc}{\chi}

\newcommand{\Gk}{\kappa}

\newcommand{\Gl}{\lambda}
\newcommand{\Gn}{\eta}
\newcommand{\Gm}{\mu}

\newcommand{\Gt}{\theta}

\newcommand{\Gr}{\rho}

\newcommand{\Gs}{\sigma}

\newcommand{\Go}{\omega}
\newcommand{\Gx}{\xi}

\newcommand{\Gz}{\zeta}
\newcommand{\GD}{\Delta}
\newcommand{\GF}{\Phi}
\newcommand{\GG}{\Gamma}
\newcommand{\GL}{\Lambda}
\newcommand{\GP}{\Pi}
\newcommand{\GT}{\Theta}

\newcommand{\GO}{\Omega}

\def\Ba{{\bf a}}
\def\Bb{{\bf b}}

\def\BA{{\bf A}}

\def\BE{{\bf E}}

\def\BH{{\bf H}}

\def\BJ{{\bf J}}

\newcommand{\beq}{\begin{equation}}
\newcommand{\eeq}{\end{equation}}
\newcommand{\barr}{\begin{eqnarray}}
\newcommand{\earr}{\end{eqnarray}}
\newcommand{\beqn}{\begin{equation*}}
\newcommand{\eeqn}{\end{equation*}}
\newcommand{\barrn}{\begin{eqnarray*}}
\newcommand{\earrn}{\end{eqnarray*}}

\newcommand{\fr}{\frac}

\begin{document}

\title[Hilbert problem for a multiply connected circular domain]{Hilbert problem for a multiply connected circular domain and the analysis of the Hall effect in a plate}
\author{Y.A. Antipov}
\address{Department of Mathematics, Louisiana State University, Baton
Rouge, Louisiana 70803}
\email{antipov@math.lsu.edu}
\thanks{The first author was supported in part by NSF Grant DMS0707724.}

\author{V.V. Silvestrov}
\address{Department of Mathematics,
Gubkin Russian State University of Oil and Gas, Moscow 119991, Russia}
\thanks{The second author was supported in part by Russian Foundation for Basic Research Grant 07-01-00038.}





\begin{abstract}

In this paper we analyze the Hilbert boundary-value  problem
of the theory of analytic functions for an $(N+1)$-connected circular domain.
An exact series-form solution has already been derived for the case
of continuous coefficients. Motivated by the study of the Hall effect
in a multiply connected plate  we extend these results by examining the case of discontinuous
coefficients. The Hilbert problem maps into the Riemann-Hilbert  
problem for symmetric piece-wise meromorphic functions invariant 
with respect to a symmetric Schottky group. The solution to this problem
is derived in terms of two analogues of the Cauchy kernel, quasiautomorphic
and quasimultiplicative kernels. The former kernel is known for any symmetry 
Schottky group. We prove the existence theorem for the second, quasimultiplicative, 
kernel for any Schottky group
(its series representation is known for the first class groups only).
We also show that the use of an automorphic kernel requires the solution to
the associated real analogue of the Jacobi inversion problem which can be bypassed
if we employ the quasiautomorphic
and quasimultiplicative kernels. We apply this theory  
to a model steady-state problem on the motion of charged electrons in a plate
with $N+1$ circular holes 
with electrodes and dielectrics on the walls when the conductor is placed 
at right angle to  the  applied magnetic field.

\end{abstract}

\maketitle

\setcounter{equation}{0}

\section{Introduction}


Let $D(\ni\infty)$ be an $(N+1)$-connected domain, a complex $z$-plane with $N+1$
holes bounded by Lyapunov contours $L_\nu$ ($\nu=0,1,\ldots,N$), and let
$a(t)$, $b(t)$, and $c(t)$ be some prescribed real functions H\"older-continuous on
the contour $L=\cup_{\nu=0}^N L_\nu$.
The second fundamental boundary-value
problem of the theory of analytic functions, the Hilbert
problem, requires
the finding of all functions $\Gf(z)$ that are single-valued and analytic in $D$,
H\"older-continuous up to the boundary $L=\cup_{\nu=0}^N L_\nu$ and satisfying
the boundary condition
\beq
\R[\ov{f(t)}\Gf(t)]=c(t), \quad t\in L,
\label{1.1}
\eeq
where $f(t)=a(t)+ib(t)$.

If at least one of the contours $L_\nu$ is not a circle, and $N\ge 1$, then 
the Hilbert problem cannot be solved exactly. In this case, by the method
of the regularizing Schwarz factor \cite{vek},
\cite{gak}, it can be reduced to a system of singular integral equations.
Alternatively, the solution can be expressed through a basis 
of $N+1$ complex harmonic measures \cite{zve}. This method, in addition to
the Schwarz factor, uses the theory of the Riemann-Hilbert problem
on a Riemann surface and requires the solution to a real analogue of the classical Jacobi
inversion problem.

In the case when all the contours $L_\nu$ are circles, and the functions $a(t)$,
$b(t)$, and $c(t)$ are H\"older-continuous on $L$, the Hilbert problem admits an exact 
solution in a series form \cite{aks}, \cite{ale}, \cite{sil}.
One of the ways to solve the problem in this case  is to convert the
original Hilbert problem into a Riemann-Hilbert problem for symmetric piece-wise
meromorphic functions invariant with respect to a Schottky group of symmetric M\"obius
transformations \cite{sil}. This idea was used in the study of steady-state flow around
$N+1$ cylinders with porous walls \cite{ant}. Its solution requires the analysis
of a Riemann-Hilbert problem with continuous coefficients. Recently \cite{ant1},   
the method was extended to free boundary problems on supercavitating flow in multiply connected domains. 
The key step in the solution procedure is the determination of a conformal
mapping in terms of the solutions to two Hilbert problems for a multiply connected
circular domain. The first problem has continuous coefficients whilst
the second one is a homogeneous problem with the coefficient
\beq
G(\Gx)=\left\{\begin{array}{cc}
-1, & \Gx\in L_j',  \\
1, & \Gx\in L_j'',  \\
\end{array}
\right. \quad j=0,1,\ldots,N,
\label{1.2}
\eeq
where $L_j=L_j'\cup L_j''$ and $L_j$ are circles.

In the present paper,
motivated by an electromagnetic problem for a Hall semiconductor 
with $N+1$ circular holes,
we analyze the general case of the Hilbert problem (\ref{1.1})
with discontinuous functions  $a(t)$, $b(t)$, and $c(t)$. 
The actual physical problem is homogeneous, and the coefficients $a(t)$ and $b(t)$
are discontinuous functions. The discontinuity is caused by the presence
of electrodes and dielectrics on the walls of the holes.
Due to the generalized Ohm's law describing the Hall effect the boundary 
conditions on the electrodes and the dielectrics, $E_\tau=0$ and $J_n=0$,
respectively, and the Maxwell equations give rise to a particular case of the Hilbert
problem with piece-wise continuous coefficients $a(t)$ and $b(t)$. Here $E_\tau$ is
the tangent component of the electric field intensity, and $J_n$ is the normal
component of the current intensity.

Various authors \cite{wic},  \cite{hae1}, \cite{hae2}, 
 \cite{ver1}, \cite{ver2}, \cite{eme} investigated the electrical characteristics of Hall plates.
These papers adopt the method of conformal maps devised by Wick \cite{wic}
for simply connected Hall plates. The method of conformal maps was further developed and
numerically implemented in  \cite{tre}, \cite{dri}
for simply connected polygonal domains. Some particular cases of a doubly connected
Hall plate in the form of an annulus with a pair of symmetric electrodes were considered
in \cite{hae1}, \cite{eme}, where 
exact solutions were derived in terms of elliptic functions.
To the knowledge of the authors, an analytical solution to
the  general case of the electromagnetic 
problem for a circular multiply connected Hall plate with any finite number
of electrodes and dielectrics on the walls is not available in the literature.

One of the steps of the solution to the Hilbert problem with discontinuous
coefficients is the factorization problem. Its solution was  derived \cite{antcro}
for an $(N+1)$-connected
circular domain 
in terms of an automorphic analogue of the Cauchy kernel. The kernel was expressed
through the Schottky-Klein prime function of the associated Schottky double. 
This procedure requires eliminating singularities at extra 
poles of the kernel by solving a real analogue of the Jacobi inversion problem
and normalizing the basis of abelian integrals.

In this paper we aim to derive 
an exact solution to the general case of the Hilbert problem (\ref{1.1})
with discontinuous coefficients 
for an $(N+1)$-connected circular domain.   
First, we shall reduce the Hilbert problem (\ref{1.1}) to the first fundamental problem of
analytic functions, the Riemann-Hilbert problem for symmetric piece-wise
meromorphic functions invariant with respect to a Schottky group.
Next, we shall introduce a multiplicative canonical function and derive its 
representation in terms of a quasiautomorphic analogue of the Cauchy kernel 
(Theorem \ref{t3.1}). It turns out that the use of a quasimultiplicative
analogue of the Cauchy kernel for the solution of both the homogeneous and
inhomogeneous problems
allows to bypass  
the Jacobi inversion problem.
Such a kernel was derived in \cite{sil} for the first 
class groups (Burnside's classification \cite{bur}). Here (Theorem \ref{t4.1}), by using 
the Riemann-Roch theorem for multiplicative functions \cite{kra}, 
we shall prove the existence of such a kernel. Then we shall derive the general solution
to the homogeneous and inhomogeneous cases of the Hilbert problem and analyze
its solvability (Theorems \ref{t5.1} and \ref{t5.2}). We shall also specify the solution for the first
class groups. In addition, we shall solve the Hilbert problem (\ref{1.1})
with discontinuous coefficients in terms of an automorphic canonical function
and the solution to the associated real analogue of the Jacobi problem.
Motivated by applications in electromagnetics
 we shall present the solution in the special case
for piece-wise constant coefficients. Finally, we shall give an exact solution
to a circular $(N+1)$-connected plate with electrodes and dielectrics
on the walls when the applied electric and transverse magnetic fields   
cause the Hall effect. The solution will be  presented in a series form
for the first class Schottky groups.

\setcounter{equation}{0}

\section{Riemann-Hilbert problem with discontinuous coefficients
for piece-wise automorphic symmetric functions}

\begin{figure}[t]
\centerline{
\scalebox{0.6}{\includegraphics{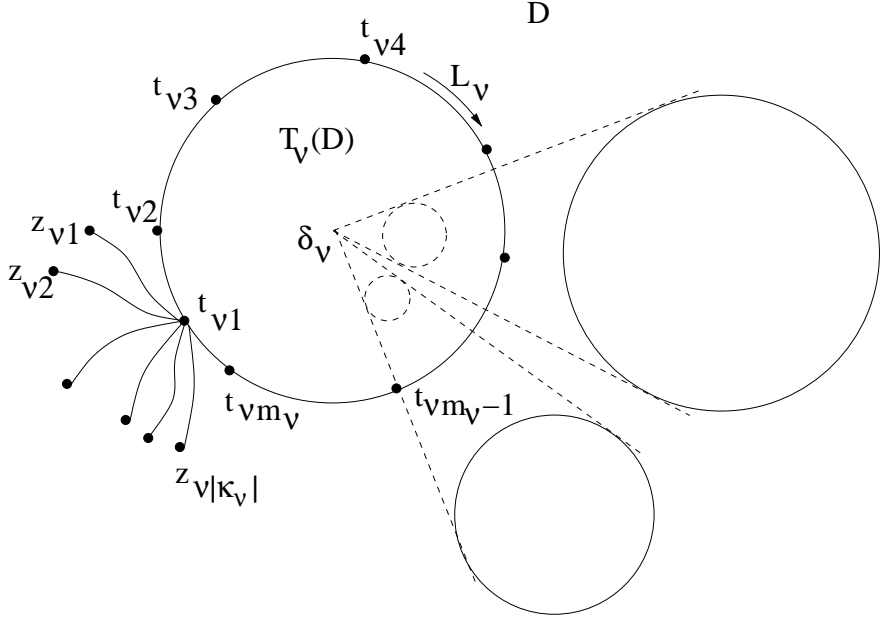}}
}
\caption{The geometry of the problem.}
\label{fig1}
\end{figure}

Let $D$ be an $(N+1)$-connected domain which is a complex 
$z$-plane with $N+1$ holes bounded by circles $L_\nu=\{z\in{\mathbb C}: |z-\Gd_\nu|=\Gr_\nu\}$,
$\nu=0,1,\ldots,N$ (Fig. 1). The positive direction on the circle $L_\nu$
is chosen such that the exterior of $L_\nu$ is on the left.
Define $L=\cup_{\nu=0}^N L_\nu$ and consider the following 
Hilbert problem:

\begin{problem}\label{p1} { Let
\beq
a(t)=a_\nu(t)\quad {\it and }\quad b(t)=b_\nu(t) \quad
(t\in L_\nu, \; \nu=0,1,\ldots,N)
\label{2.1}
\eeq
be real functions satisfying the H\"older condition either everywhere on 
$L_\nu$, or everywhere on $L_\nu$ except at points $t_{\nu 1}$, 
$t_{\nu 2}$, $\ldots$, $t_{\nu m_\nu}$, $\nu=0,1,\ldots,N,$
where at least one of the functions (\ref{2.1}) has a jump 
discontinuity. Assume that $a_\nu^2(t)+b_\nu^2(t)\ne 0$ everywhere
on all the contours $L_\nu$. Let $c(t)=c_\nu(t)$
($t\in L_\nu$) be a real H\"older-continuous function on  
$L_\nu$, $\nu=0,1,\ldots,N$.  

Find all functions $\Gf(z)=u(z)+iv(z)$, holomorphic in $D$, H\"older-continuous
everywhere in $D\cup L$ except at the points $t_{\nu j}$
($j=1,2,\ldots, m_\nu$, $\nu=0,1,\ldots,N$),
where they may have integrable singularities, bounded at infinity and
satisfying the boundary condition
\beq
a(t)u(t)+b(t)v(t)=c(t), \quad t\in L\setminus{\GT},
\label{2.2}
\eeq
where $\GT=\cup_{\nu=0}^N\cup_{j=1}^{m_\nu}t_{\nu j}$.}
\end{problem}

To solve this problem, we transform it into a Riemann-Hilbert 
problem for piece-wise automorphic functions.  
For this, we consider 
the symmetry group, $\frak G$, of the line $L=L_0\cup L_1\cup\ldots\cup L_N$
generated by the linear transformations
$\Gs_\nu(z)=T_\nu T(z)$, $\nu=1,2,\ldots,N$, where 
\beq
T=T_0, \quad T_\nu(z)=\Gd_\nu+\fr{\Gr_\nu^2}{\bar z-\bar\Gd_\nu}, \quad \nu=0,1,\ldots, N,
\label{2.2a}
\eeq
and $T_\nu$ is the symmetry transformation with respect to the circle $L_\nu$.
Denote the fundamental region of the group $\frak G$ by $\frak F=D\cup T(D)\cup L$.
The group $\frak G$ is a symmetry Schottky group \cite{for}. The elements of the 
group are
the identical map $\Gs_0(z)=z$ and all possible
compositions of the generators $\Gs_\nu=T_\nu T$ and the inverse maps
$\Gs_\nu^{-1}=TT_\nu$ ($\nu=1,2,\ldots,N$):
$$
\Gs\in\frak G \Leftrightarrow\Gs=T_{k_{2\mu}} T_{k_{2\mu-1}}\ldots T_{k_{2}}T_{k_{1}},\quad \mu=1,2,\ldots,
$$
\beq
k_1,k_2,\ldots,k_{2\mu}=0,1,\ldots,N, \quad
k_2\ne k_1, k_3\ne k_2, \ldots, k_{2\mu}\ne k_{2\mu-1}.
\label{2.2b}
\eeq
The region $\frak D=\cup_{\Gs\in\frak G}\Gs(\frak F)$ is invariant with respect to the group
$\frak G$: $\Gs(\frak D)=\frak D$ for all $\Gs\in \frak G$. This region is symmetric with
respect to all the circles $L_\nu$ ($\nu=0,1,\ldots,N$), and
$\frak D=\bar{\mathbb C}\setminus \GL$, where $\bar{\mathbb C}={\mathbb C}\cup\{\infty\}$ is
the extended $z$-plane,
and $\GL$ is the set of the limit points of the group $\frak G$ (it consists of
two points if $N=1$ and it is infinite if $N\ge 2$).
Notice that all elements of 
the group $\frak G$ can be represented in the form
\beq
\Gs(z)=\fr{a_\Gs z+b_\Gs}{c_\Gs z+d_\Gs},\quad a_\Gs d_\Gs-b_\Gs c_\Gs\ne 0,
\label{2.2c}
\eeq
and $c_\Gs\ne 0$ if $\Gs\ne \Gs_0$.

Introduce now a new function
\beq
\GF(z)=\left\{
\begin{array}{cc}
\Gf(z), & z\in D,\\
\ov{\Gf(T(z))}, & z\in T(D).\\
\end{array}\right.
\label{2.3}
\eeq
Extend next the definition of the function $\GF(z)$ from the domain $D\cup T(D)$ 
into the region $\frak D$ by the automorphicity law, 
\beq
\GF(z)=\GF(\Gs^{-1}(z)), \quad z\in\Gs(D\cup T(D)), \quad \Gs\in \frak G\setminus\Gs_0.
\label{2.4}
\eeq
The function $\GF(z)$ so defined is a 
piece-wise meromorphic function with the discontinuity line
$\frak L=\cup_{\Gs\in\frak G}\Gs(L)$
invariant with respect to the group $\frak G$:
\beq
\GF(\Gs(z))=\GF(z), \quad z\in\frak D\setminus\frak L,\quad \Gs\in\frak G.
\label{2.4'}
\eeq
In addition, the function $\GF(z)$ satisfies the symmetry condition
\beq
\ov{\GF(T_\nu(z))}=\GF(z), \quad z\in\frak D\setminus\frak L, \quad \nu=0,1,\ldots,N,
\label{2.5}
\eeq
which follows from (\ref{2.3}) and (\ref{2.4}).
In order to write the boundary condition (\ref{2.2}) in terms of the function $\GF(z)$,
let $z\to t\in L_\nu$, $z\in D$. Then $T_\nu(z)\to t$, $T_\nu(z)\in T_\nu(D)$.
Introduce the following notations
$$
\GF^+(t)=\lim_{z\to t, \; z\in D}\GF(z)=\Gf(t),
$$
\beq
\GF^-(t)=\lim_{z\to t, \; z\in T_\nu(D)}\GF(z)=\lim_{T_\nu(z)\to t, \; T_\nu(z)\in D}
\ov{\GF(T_\nu(z))}=\ov{\Gf(t)}.
\label{2.6}
\eeq
Now inspection of the boundary condition (\ref{2.2}) shows that it is essentially equivalent to the equation
\beq
\R\{[a(t)-ib(t)]\Gf(t)\}=c(t), \quad t\in L\setminus\GT,
\label{2.7}
\eeq
or, equivalently, in terms of the functions (\ref{2.6}),
\beq
\GF^+(t)=p(t)\GF^-(t)+q(t), \quad t\in L\setminus\GT,
\label{2.8}
\eeq
where
\beq
p(t)=-\fr{a(t)+ib(t)}{a(t)-ib(t)},
\quad q(t)=\fr{2c(t)}{a(t)-ib(t)}.
\label{2.9}
\eeq


\begin{definition}\label{d1}

We say that a function $\GF(z)\in Q_\frak G(\frak L)$ if it is piece-wise
meromorphic with the discontinuity line $\frak L$, invariant with respect to the group $\frak G$:
$\GF(\Gs(z))=\GF(z)$, $\Gs\in\frak G$, $z\in\frak D\setminus\frak L$, and $T$-symmetric:
$\ov{\GF(T(z))}=\GF(z)$, $z\in\frak D\setminus\frak L$.      
\end{definition}

The fact that the boundary values of the function $\GF(z)$ satisfy the condition (\ref{2.8})
lends itself to the opportunity of stating the following Riemann-Hilbert boundary-value problem
with discontinuous coefficients in the class of functions $Q_\frak G(\frak L)$:


\begin{problem}\label{p2} { Find all functions $\GF(z)\in Q_\frak G(\frak L)$, H\"older-continuous
in the domain $\frak D\cup\frak L$ apart from the set of points $\Gs(\GT)$, $\Gs\in\frak G$,
where they may have integrable singularities, bounded at  the points $\Gs(\infty)$ and
which satisfy the boundary condition (\ref{2.8}).}
\end{problem}

\setcounter{equation}{0}

\section{Multiplicative canonical function}

Let $t_{\nu 1}$ be the starting point of the circle $L_{\nu}$. This point can be 
chosen arbitrarily if both the functions $a(t)$ and $b(t)$ are continuous
everywhere on the circle $L_\nu$. 

\begin{definition}\label{d2} We say that a function $\Gc(z)$ is a  
{\it multiplicative canonical function} of Problem \ref{p2} if 

(i) it is a piece-wise meromorphic function in the region $\frak D$ with the discontinuity line $\frak L$,
H\"older-continuous in the domain $\frak D\cup\frak L$ except for the points
$\Gs(t_{\nu 1})$, where it may have a power singularity of any finite exponent,
and the points $\Gs(t_{\nu 2})$, $\Gs(t_{\nu 3})$,
$\ldots, \Gs(t_{\nu m_{\nu}})$, $\Gs\in\frak G$
where it may have integrable singularities,

(ii) its boundary values $\Gc^\pm(t)$ satisfy the boundary
condition 
\beq
\Gc^+(t)=p(t)\Gc^-(t), \quad t\in L\setminus\GT,
\label{3.0}
\eeq

(iii) it is a $T$-symmetric function: $\Gc(z)=\ov{\Gc(T(z))}$, $z\in \frak D\setminus\frak L$, and

(iv) it satisfies the multiplicativity condition
$\Gc(\Gs(z))=H_\Gs^{-1}\Gc(z)$, $\Gs\in\frak G$, $z\in\frak D\setminus\frak L$,
with the character $H^{-1}$,  where $H$ is a group homomorphism
between $\frak G$ and a multiplicative group $\frak H$ of complex numbers such that
$H_{\Gs\Go}=H_\Gs H_\Go$.

\end{definition}

To find such a function, we will use a quasiautomorphic analogue
of the Cauchy kernel. Prove first its existence.  

\begin{theorem}\label{t3.0} There exists a function
$K(z,\tau)$ which has the following properties:

(i) for each fixed $\tau\in L$, $K(z,\tau)=\fr{1}{\tau-z}+B(z,\tau)$,
where $B(z,\tau)$ is an analytic function of $z\in\frak F$,

(ii) there exists a point $z_*\in\frak F$, such that $K(z_*,\tau)=0$ for all $\tau\in L$,

(iii) for any $\Gs\in\frak G$,
\beq
K(\Gs(z),\tau)=K(z,\tau)+\Gn_\Gs(\tau),
\label{3.1}
\eeq
where $\Gn_\Gs(\tau)=K(\Gs(z_*),\tau)$.

\end{theorem}

\begin{proof}
The existence of such a function for any discrete discontinuous group
of M\"obius transformations and, in particular, for a Schottky symmetry group,
follows from the theory of abelian integrals on closed Riemann surfaces \cite{hur}.
Indeed, the fundamental region $\frak F$ becomes a closed Riemann surface of genus $N$
if we add the circles $L_\nu'=\Gs_\nu^{-1}(L_\nu)$ and consider
all congruent points of the circles $L_\nu$ and $L_\nu'$ ($\nu=1,2,\ldots,N$)
as identical. The cycles $L_\nu$ may be accepted as canonical cross-sections $\Ba_\nu$,
and any simple curve joining a pair of congruent points $\Gx_\nu'\in L_\nu'$ and 
$\Gs_\nu(\Gx_\nu')\in L_\nu$ can be considered as a canonical cross-section $\Bb_\nu$. As a function
of $z$, the kernel $K(z,\tau)$ is an abelian integral of the second kind.  It has only one
simple pole at the point $z=\tau$ with the residue $-1$. It vanishes at the point $z=z_*$ 
and has zero $A$-periods,
$
\oint_{a_\nu}d_z K(z,\tau)=0.   
$ 
It is known \cite{hur} that such an abelian integral exists and it is unique.
This integral has non-zero $B$-periods, 
\beq
\oint\limits_{b_\nu}d_zK(z,\tau)=K(\Gs_\nu(\Gx'_\nu),\tau)-K(\Gx_\nu',\tau)
=\eta_\nu(\tau),
\quad \nu=1,2,\ldots,N,
\label{3.3}
\eeq
where $\Gn_\nu(\tau)=K(\Gs_\nu(z_*),\tau)$, and the maps $\Gs_\nu=T_\nu T$ are generating transformations of the group $\frak G$.
\end{proof}
 
Note that the functions $\eta_\nu(\tau)$ ($\nu=1,2,\ldots, N$) are linearly independent, and the differentials
$\fr{1}{2\pi i}\eta_\nu(\tau)d\tau$ ($\nu=1,2,\ldots, N$) form the normalized basis of abelian differentials
of the fist kind on the Riemann surface,
$$
\fr{1}{2\pi i}\int\limits_{L_\nu}\eta_k(\tau)d\tau =\left\{\begin{array}{cc}
1, & k = \nu,\\
0, & k\ne \nu,\\
\end{array}
\right.
$$
\beq
\fr{1}{2\pi i}\int\limits_{T(t_{\nu 1})}^{t_{\nu 1}} \eta_k(\tau)d\tau= B_{k\nu}.
\label{3.4}
\eeq
The matrix of  $B$-periods, $||B_{k\nu}||$, $k,\nu=1,2,\ldots, N$, is symmetric and its imaginary part
is positive definite.

\begin{definition}\label{d3} A function  $K(z,\Gx)$ is said to be
a {\it quasiautomorphic analogue}
of the Cauchy kernel if it possesses properties (i) to (iii) of Theorem \ref{t3.0}.
\end{definition}

\begin{remark}\label{r1}

Because of the property (i), the integral
\beq
\fr{1}{2\pi i}\int\limits_L\Gvf(\tau)K(z,\tau)d\tau
\label{3.2}
\eeq
satisfies the Sokhotski-Plemelj formulas. 

\end{remark}

If the group $\frak G$ is of the first class \cite{bur} 
or, equivalently, the numerical series
\beq
\sum_{\Gs\in\frak G\setminus\Gs_0}\fr{|a_\Gs d_\Gs-b_\Gs c_\Gs|}{|c_\Gs|^2}
\label{3.5}
\eeq
is convergent, such a kernel is known \cite{chi}, 
\beq
K(z,\tau)=\sum_{\Gs\in \frak G}\left(
\fr{1}{\Gs(\tau)-z}-\fr{1}{\Gs(\tau)-z_*}
\right)\Gs'(\tau)=
\sum_{\Gs\in\frak G}\left(
\fr{1}{\tau-\Gs(z)}-\fr{1}{\tau-\Gs(z_*)}\right).
\label{3.6}
\eeq
In general, the kernel can be expressed through the Schottky-Klein prime function
$\Go(z,\tau)$ associated with the group $\frak G$ \cite{antcro} by the formula
\beq
K(z,\tau)=\fr{d}{d\tau}\ln\left(\fr{\Go(z,\tau)}{\Go(z_*,\tau)}\right).
\label{3.7}
\eeq

Define next the logarithmic function $\ln p(\tau)$ or, equivalently, $\arg p(\tau)$ on each arc $t_{\nu j} t_{\nu j+1}$,
$\nu=0,1,\ldots,N$, $j=1,2,\ldots m_\nu$. We shall use the definitions $t_{\nu j}^+$ and $t_{\nu j+1}^-$
to indicate the starting and the terminal points of the arc $t_{\nu j}t_{\nu j+1}$, respectively  
(it is assumed that $t_{\nu m_\nu+1}=t_{\nu 1}$).
On the arc $t_{\nu 1} t_{\nu 2}$, a branch of the function $\arg p(\tau)$
can be fixed arbitrarily. We fix it by the condition 
\beq
-\pi<\arg p(t_{\nu 1}^+)\le\pi.
\label{3.8}
\eeq
Let $\GD_{\nu j}$ be the change of $\arg p(\tau)$ along the
arc  $t_{\nu j} t_{\nu j+1}$ ($j=1,2,\ldots, m_\nu$),
\beq
\GD_{\nu j}=[\arg p(\tau)]_{t_{\nu j}t_{\nu j+1}}. 
\label{3.9}
\eeq
Then, obviously,
$\arg p(t_{\nu j+1}^-)=\arg p(t_{\nu j}^+)+\GD_{\nu j}$, $j=1,2,\ldots,m_\nu$. The values 
$\arg p(t_{\nu j}^+)$ ($j=2,3,\ldots,m_\nu$) cannot be chosen arbitrarily. Since the solution $\GF(z)$ may have integrable singularities
at the points  $t_{\nu j}$, define a continuous branch of the function $\arg p(\tau)$ by 
\beq
-2\pi<\arg p(t_{\nu j}^-)-\arg p(t_{\nu j}^+)\le 0, \quad j=2,\ldots,m_\nu.
\label{3.10}
\eeq   
We next choose integers $\Gk_\nu$ such that
\beq
-4\pi<\arg p(t_{\nu 1}^-)-\arg p(t_{\nu 1}^+)-4\pi\Gk_\nu\le 0, \quad \nu=0,1,\ldots,N,
\label{3.11}
\eeq
and prove the following result.

\begin{theorem}\label{t3.1} Let 
\beq
\GG(z)=\fr{1}{4\pi}\int\limits_L
\arg p(\tau) K(z,\tau)d\tau +\sum_{\nu=0}^N\sgn\Gk_\nu\sum_{j=1}^{|\Gk_\nu|}\int\limits_{\Gg_{\nu j}}K(z,\tau)d\tau,  
\label{3.12}
\eeq
where $\Gg_{\nu j}=t_{\nu 1}z_{\nu j}$
are piece-wise smooth curves in $D$ which do not cross each other, and
 $z_{\nu j}$ are arbitrarily fixed distinct points in the region $D$ (Fig. 1).
Then 
\beq
\Gc(z)=\exp\{\GG(z)+\ov{\GG(T(z))}\}
\label{3.14}
\eeq
is a multiplicative canonical function.
\end{theorem} 

\begin{proof}
 Analyze first the behavior of the function  $\Gc(z)$
at the points of the set $\GT$. Clearly, in a neighborhood of the point
$z=t_{\nu j}$,
\beq
\GG(z)=\Ga_{\nu j}\ln(z-t_{\nu j})+ f_0(z), \quad j=1,2,\ldots,m_\nu, \quad \nu=0,1,\ldots,N,
\label{3.15}
\eeq
where $f_0(z)$ is a function bounded as $z\to t_{\nu j}$,
$$
\Ga_{\nu 1}=\fr{1}{4\pi}[\arg p(t_{\nu 1}^-)-\arg p(t_{\nu 1}^+)]-\Gk_{\nu},
$$
\beq
\Ga_{\nu j}=\fr{1}{4\pi}[\arg p(t_{\nu j}^-)-\arg p(t_{\nu j}^+)],\quad j=2,3,\ldots,m_\nu.
\label{3.16}
\eeq
Since $T_\nu(t_{\nu j})=t_{\nu j}$ we conclude from (\ref{2.2a}) that
\beq
T_\nu(z)-t_{\nu j}\sim -\fr{\Gr_\nu^2}{(\bar t_{\nu j}-\bar\Gd_\nu)^2}(\bar z-\bar t_{\nu j}),
\quad z\to t_{\nu j}.
\label{3.17}
\eeq
On the other hand,
$$
\GG(T(z))=\GG(\Gs_\nu^{-1}(T_\nu(z)))=\GG(T_\nu(z))-\GG(\Gs_\nu(z_*))
$$
\beq
=\Ga_{\nu j}\ln(T_\nu(z)-t_{\nu j})+f_1(z), \quad
z\to t_{\nu j},
\label{3.18}
\eeq
where $f_1(z)$ is a function bounded as $z\to t_{\nu j}$.
From the definition of the function $\Gc(z)$ (\ref{3.14}) and (\ref{3.12}), it follows that
$$
\Gc(z)\sim A_{\nu j}(z-t_{\nu j})^{2\Ga_{\nu j}},\quad
A_{\nu j}=\const\ne 0,  \quad z\to t_{\nu j},
$$
\beq
j=1,2,\ldots,m_\nu,\quad
\nu=0,1,\ldots,N.
\label{3.19}
\eeq
Since $-\fr12<\Ga_{\nu j}\le 0$ ($j\ne 1$), the function
$\Gc(z)$ may have at most an integrable singularity as $z\to t_{\nu j}$  ($j\ne 1$).
At the point $z=t_{\nu 1}$, the function $\Gc(z)$ has an integrable singularity
if $-\fr12<\Ga_{\nu 1}\le 0$ and a nonintegrable singularity of order $1\le -2\Ga_{\nu 1}<2$
if $-1<\Ga_{\nu 1}\le -\fr12$. 

Analysis of the second term in (\ref{3.12}) implies that if $\Gk_\nu\ne 0$, then at the points $z_{\nu j}$,
the function $\Gc(z)$ 
has a simple pole provided $\Gk_\nu$ is negative and a simple zero provided
$\Gk_\nu$ is positive. Apart from these points, the function $\Gc(z)$ 
is analytic 
everywhere in the region $D$ and does not vanish. 
In the case $\Gk_\nu=0$, $z_{\nu j}$ are regular points of the function $\Gc(z)$.

Verify now that the boundary values, $\Gc^+(t)$ and $\Gc^-(t)$,
of the function $\Gc(z)$ satisfy the linear relation (\ref{3.0}). By applying the Sokhotski-Plemelj
formulas to the integral
\beq
\GG_0(z)=\fr{1}{4\pi i}\int\limits_L \ln p(\tau)K(z,\tau)d\tau,
\label{3.20}
\eeq
and noticing that $|p(\tau|=1$, we obtain 
\beq
\GG_0^+(t)-\GG_0^-(t)=\fr{i}{2}\arg p(t), \quad t\in L\setminus\GT.
\label{3.21}
\eeq
Consider now $\GG_0^\pm(T(t))$, $t\in L$.
Let first $t\in L_0$ and $z\to t^\pm$. Clearly, then $T(z)\to t^\mp$ and
$\GG_0^\pm(T(t))=\GG_0^\mp(t)$. This implies
\beq
\GG_0^+(T(t))-\GG_0^-(T(t))=-\fr{i}{2}\arg p(t).
\label{3.22}
\eeq
For $t\in L_\nu$ ($\nu=1,2,\ldots,N$), because of the identity
\beq
\GG_0(T(z))=\GG_0(T_\nu(z))-\GG_0(\Gs_\nu(z_*)),
\label{3.23}
\eeq
the one-sided limits $\GG_0^\pm(T_0(t))$ of the function $\GG_0(T(z))$ 
meet the condition (\ref{3.22}). Therefore, the jump of the
function $\GG_0(z)+\ov{\GG_0(T(z))}$, when $z$ passes through the line $L$,
equals $i\arg p(t)=\ln p(t)$. 

Notice that the function $\GG(z)$ is discontinuous when $z$ passes
through the curves  $t_{\nu 1}z_{\nu j}$, and the jump is a multiple of $2\pi i$.
This means that the function $\Gc(z)$ itself is continuous through these curves.

We observe next that the function $\Gc(z)$ is a $T$-symmetric function: $\ov{\Gc(T(z))}=\Gc(z)$,
$z\in \frak D\setminus \frak L$. To finalize the proof of the theorem, we need to show that
$\Gc(z)$ is a multiplicative function. 
The property (\ref{3.1}) of the kernel $K(z,\tau)$ written for the 
generating transformations $\Gs_\nu(z)$ implies 
$$
\GG(\Gs_\nu(z))=\GG(z)+h_\nu,
$$
\beq
\GG(T(\Gs_\nu(z))=\GG(TT_\nu T(z))=\GG(\Gs_\nu^{-1}T(z))=\GG(T(z))-h_\nu,
\label{3.24}
\eeq
where $h_\nu=\GG(\Gs_\nu(z_*))$, $\nu=1,2\ldots,N.$ Therefore,
$$
\Gc(\Gs_\nu(z))=H_\nu^{-1}\Gc(z), \quad z\in\frak D\setminus\frak L, 
$$
\beq
H_\nu=\exp(-2i\I h_\nu), \quad \nu=1,2,\ldots, N.
\label{3.25}
\eeq
Consider now the general form of the transformation $\Gs=T_{\nu_{2\mu}}T_{\nu_{2\mu-1}}
\ldots T_{\nu_{2}}T_{\nu_{1}}$. It can also be written in the form
\beq
\Gs=\Gs_{\nu_{2\mu}}\Gs_{\nu_{2\mu-1}}^{-1}\ldots
\Gs_{\nu_{2}}\Gs_{\nu_{1}}^{-1}, \quad \mu=1,2,\ldots.
\label{3.26}
\eeq
Since
\beq
\Gc(\Gs^{-1}_\nu(z))=H_\nu\Gc(z), \quad z\in\frak D\setminus\frak L, 
\label{3.27}
\eeq
the property (\ref{3.25}) for the generating transformations is valid for any transformation
$\Gs\in\frak G$ provided the number $H_\nu$ is replaced by $H_\Gs$,
$$
\Gc(\Gs(z))=H_\Gs^{-1}\Gc(z), \quad z\in\frak D\setminus\frak L, 
$$
\beq
H_\Gs=\fr{H_{\nu_2}H_{\nu_4}\ldots H_{\nu_{2\mu}}}
{H_{\nu_1}H_{\nu_3}\ldots H_{\nu_{2\mu-1}}}, \quad H_{\Gs_0}=1.
\label{3.28}
\eeq
The one-to-one map (\ref{3.28}), $H$, from the group $\frak G$ into a multiplicative group
$\frak H$
of complex numbers $H_\Gs$, $\Gs\in\frak G$, has the following property 
\beq
H_{\Gs\Go}=H_\Gs H_\Go \quad \forall \Gs,\Go\in\frak G.
\label{3.29}
\eeq
Thus, $H$ is a homomorphism between these two groups, and $\Gc(z)$ is a multiplicative
function with the character $H^{-1}$ \cite{kra}.
\end{proof}

\setcounter{equation}{0}
\section{Quasimultiplicative analogue of the Cauchy kernel}

To solve the homogeneous case ($q(t)\equiv 0$) of Problem \ref{p2},
we need a  quasimultiplicative analogue of the Cauchy kernel.

\vspace{.1in}

\begin{theorem}\label{t4.1} There exists a function
$M(z,\tau)$ which has the following properties:

(i) for each fixed $\tau\in L$, $M(z,\tau)=\fr{1}{\tau-z}+B_0(z,\tau)$,
where $B_0(z,\tau)$ is an analytic function of $z\in\frak F$,

(ii) there exists a point $z_0\in\frak F$, such that $M(z_0,\tau)=0$ for all $\tau\in L$,

(iii) for each fixed $\tau\in L$ and for any $\Gs\in\frak G$, there exists a function $\Gz_\Gs(\tau)$ such that
\beq
M(\Gs(z),\tau)=H_\Gs M(z,\tau)+\Gz_\Gs(\tau).
\label{4.0}
\eeq
\end{theorem}

\begin{definition}\label{4} A function  $M(z,\Gx)$ which possesses properties (i) to (iii) is said to be
a {\it quasimultiplicative analogue}
of the Cauchy kernel with the character $H$.
\end{definition}

\begin{proof} Introduce a function $P(z)=\fr{\Md M(z,\tau)}{\Md z}$. From the conditions
(i) and (iii) it follows that
\beq
P(z)=\fr{1}{(\tau-z)^2}+B_1(z,\tau) {\;\rm for\; each\; fixed\;} \tau\in L, 
\label{4.1}
\eeq
where $B_1(z,\tau)=\fr{\Md B_0(z,\tau)}{\Md z}$, and also
\beq
\Gs'(z)P(\Gs(z))=H_\Gs P(z), \quad \Gs\in\frak G,
\label{4.2}
\eeq
or, equivalently, $\Gs'_\nu(z)P(\Gs_\nu(z))=H_\nu P(z)$, $\nu=1,2,\ldots,N$.
This means that $P(z)$ is a multiplicative automorphic form of weight (dimension)
(-2) belonging to the character $H$ \cite{kra}. This form has only one
singularity in the fundamental region $\frak F$, a 
pole of the second order at the point $z=\tau$. At the infinite point, $z=\infty$,
it has a zero of the second order. In what follows we prove that such a form exists.

Let $\frak R$ be a Riemann surface formed by gluing the congruent sides
$L_\nu$ and $L_\nu'=\Gs^{-1}_\nu(L_\nu)$ of the fundamental region $\frak F$.
Choose the canonical cross-sections of the surface $\frak R$  (the canonical homology basis on $\frak R$)
as follows: $\Ba_\nu=L_\nu$ and $\Bb_\nu=\Gx_\nu'\Gx_\nu$ with
$\Gx_\nu'\in L'_\nu$ and $\Gx_\nu=\Gs_\nu(\Gx_\nu')\in L_\nu$.
On the surface $\frak R$, the differential $dP^\circ(z)=P(z)dz$ can be interpreted
as a multiplicative differential with the character $H$ defined by
\beq
H[\Ba_\nu]=1, \quad H[\Bb_\nu]=H_\nu, \quad \nu=1,2,\ldots,N.
\label{4.3}
\eeq
Show next that there exist exactly $N+1$ linearly independent multiplicative differentials
which, on the surface $\frak R$, have only one pole of multiplicity not higher than 2. Let
$r_*$ be the dimension of the space $\frak M_H(\frak d_*)$ of multiplicative differentials with the character $H$
whose divisors $\frak d_*$ are multiples
of the divisor $\frak d_0=\tau^{-2}$, and let $r$ be the dimension of the space $\frak M_{H^{-1}}(\frak d_*^{-1})$
of multiplicative functions 
with the character $H^{-1}$  whose divisors $\frak d_*^{-1}$
are multiples
of the divisor $\frak d_0^{-1}=\tau^{2}$. The character $H^{-1}$ is 
defined by the factors $H^{-1}[\Ba_\nu]=1$,  $H^{-1}[\Bb_\nu]=H^{-1}_\nu$,
$ \nu=1,2,\ldots,N$. 
The space $\frak M_{H^{-1}}(\frak d_*^{-1})$ consists of multiplicative functions which have
a second-order zero at the point $z=\tau$ and which do not have any singularities on $\frak R$. Clearly,
the dimension of this space, $r$, is zero.
By the Riemann-Roch theorem for multiplicative functions on a genus-$N$ Riemann surface $\frak R$ \cite{far}, p.126,
\beq
r_*=r-\deg \frak d_0+N-1.
\label{4.4}
\eeq
Since $r=0$ and $\deg\frak d_0=-2$ we find $r_*=N+1$. Similarly, the dimension of the space 
of multiplicative differentials whose divisors are multiples
of the divisor $\frak d_1=\tau^{-1}$, equals $N$. Note, however, that there does not exist 
a differential $dP^\circ$ with factors $H[\Ba_\nu]=1$ analytic everywhere on the surface $\frak R$
except at the point $z=\tau$, 
\beq
dP^\circ(z)\sim \fr{dz}{\tau-z}.
\label{4.5}
\eeq
Otherwise, on one hand,
\beq
\int\limits_{\Md \frak F}d P^\circ=-2\pi i.
\label{4.6}
\eeq
On the other hand, 
\beq
\int\limits_{\Md \frak F}d P^\circ=\sum_{\nu=1}^N\left(\int\limits_{L_\nu}dP^\circ
-\int\limits_{L_\nu'}dP^\circ
\right)=\sum_{\nu=1}^N(1-H^{-1}[\Bb_\nu])\int\limits_{L_\nu}dP^\circ=0
\label{4.7}
\eeq
since 
\beq
\int_{L_\nu}dP^\circ=P^\circ(\Gx_\nu^-)-P^\circ(\Gx_\nu^+)=(H[\Ba_\nu]-1)P^\circ(\Gx_\nu^+)=0,
\label{4.8}
\eeq
that is in contradiction with (\ref{4.6}). 
This means there exist $N$ linearly-independent multiplicative differentials with factors 
$H[\Ba_\nu]=1$ and  $H[\Bb_\nu]=H_\nu$,
$ \nu=1,2,\ldots,N$, analytic everywhere on the surface $\frak R$. Since $r_*=N+1$, there exists at least
one multiplicative differential $dP^\circ(z)$ with factors 
$H[\Ba_\nu]=1$ and  $H[\Bb_\nu]=H_\nu$ and which has a single second-order pole with the principal part
$\fr{dz}{(\tau-z)^2}$. Then the function 
$P(z)=\fr{dP^\circ(z)}{dz}$ satisfies the conditions    
(\ref{4.1}) and (\ref{4.2}), has a second-order zero at the infinite point $z=\infty$
and also has the following property:
\beq
\int\limits_{L_\nu}P(z)dz=\int\limits_{L_\nu}dP^\circ=0
\label{4.9}
\eeq
Define now a function
\beq
M(z,\tau)=\int\limits_{z_0}^z P(\Gx)d\Gx.
\label{4.10}
\eeq
This function, as a function of $z$ and for any fixed $\tau\in L$, 
is analytic everywhere in the region $\frak F$ except at  the point $z=\tau$,
where it has a simple pole with the residue -1. 
Clearly, $M(z_0,\tau)=0$. Finally, from (\ref{4.2}),
$$
M(\Gs(z),\tau)=\int\limits_{z_0}^{\Gs(z_0)}P(\Gx)d\Gx+\int\limits_{\Gs(z_0)}^{\Gs(z)}P(\Gx)d\Gx
$$
\beq
=M(\Gs(z_0),\tau)+\int\limits_{z_0}^z P(\Gs(\Gx))d\Gs(\Gx)=H_\Gs M(z,\tau)+\Gz_\Gs(\tau), \quad \Gs\in\frak G,
\label{4.11}
\eeq
where $\Gz_\Gs(\tau)=M(\Gs(z_0),\tau)$. Because of the condition (\ref{4.9}) $M(z,\tau)$ is
a single-valued function in the fundamental region $\frak F$. The quasimultiplicativity property (\ref{4.11}) 
implies that the function $M(z,\tau)$ is single-valued everywhere in the region $\frak D$.
Thus, all the three properties (i) to (iii) have been verified that proves the existence of the kernel $M(z,\tau)$.
\end{proof}

If the group $\frak G$ is of the first class, and $|H_\Gs|=1$ for all $\Gs\in\frak G$ or, equivalently, 
$|H_\nu|=1$, $\nu=1,2,\ldots, N$, then the kernel $M(z,\tau)$ can be found explicitly \cite{sil}
through all the transformations of the group $\frak G$ in the form of 
an absolutely and uniformly convergent series
\beq
M(z,\tau)=
\sum_{\Gs\in\frak G}\fr{1}{H_\Gs}\left(
\fr{1}{\tau-\Gs(z)}-\fr{1}{\tau-\Gs(z_0)}\right).
\label{4.12}
\eeq
This kernel vanishes at $z=z_0$ and possesses the other two 
properties of the kernel $M(z,\tau)$. Indeed, because $\frak G\ni\Gs_0$,
\beq
M(z,\tau)=\fr{1}{\tau-z}-\fr{1}{\tau-z_0}+
\sum_{\Gs\in \frak G\setminus\Gs_0}\fr{1}{H_\Gs}\left(
\fr{1}{\tau-\Gs(z)}-\fr{1}{\tau-\Gs(z_0)}\right).
\label{4.13}
\eeq
To verify the property (\ref{4.0}), we employ 
the multiplicativity of $H$, $H_{\Gs\Go}=H_\Gs H_\Go$,  the relation
$$
M(\Gs(z),\tau)=\sum_{\Go\in \frak G}\fr{1}{H_\Go}\left(
\fr{1}{\tau-\Go\Gs(z)}-\fr{1}{\tau-\Go\Gs(z_0)}\right)
$$
\beq
+\sum_{\Go\in \frak G}\fr{1}{H_\Go}\left(
\fr{1}{\tau-\Go\Gs(z_0)}-\fr{1}{\tau-\Go(z_0)}\right),
\label{4.14}
\eeq
and make the substitution $\nu=\Go\Gs$. This ultimately gives
the relation wanted $M(\Gs(z),\tau)=H_\Gs  M(z,\tau)+\Gz_\Gs(\tau)$, where $\Gz_\Gs(\tau)=M(\Gs(z_0),\tau)$.

\setcounter{equation}{0}

\section{Solution to the Riemann-Hilbert problem}

Having now equipped with two analogues of the Cauchy kernel, 
the quasiautomorphic and quasimultiplicative kernels $K(z,\tau)$ and
$M(z,\tau)$, we solve Problem \ref{p2}. We begin with the homogeneous
Riemann-Hilbert problem.

\subsection{Homogeneous case: $q(t)\equiv 0$}

\begin{problem}\label{p3} {\it Find all functions $\GF(z)\in Q_\frak G(\frak L)$, H\"older-continuous everywhere 
in the domain $\frak D\cup\frak L$ apart from the set of points $\Gs(\GT)$, $\Gs\in\frak G$,
where they may have integrable singularities, bounded at the points $\Gs(\infty)$ and
satisfying the boundary condition 
\beq
\GF^+(t)=p(t)\GF^-(t), \quad t\in L\setminus\GT.
\label{5.1}
\eeq}
\end{problem}

By using (\ref{3.0}), we write the boundary condition (\ref{5.1}) in the form
\beq
\fr{\GF^+(t)}{\Gc^+(t)}=\fr{\GF^-(t)}{\Gc^-(t)}, \quad t\in L\setminus\GT,
\label{5.2}
\eeq
and analyze the function $\GF(z)/\Gc(z)$. It is a $T$-symmetric and $\frak G$-multiplicative
function with the factors $H_\Gs$, $\Gs\in\frak G$. If $\Gk_\nu>0$, then it has simple poles  
at the points $z_{\nu j}$ and $T(z_{\nu j})$ 
($j=1,2,\ldots, m_\nu$).
In the case $\Gk_\nu< 0$, the function $\GF(z)/\Gc(z)$ has simple zeros
 at the points $z_{\nu j}$ and $T(z_{\nu j})$. Let
\beq
\Gk_\nu^+=\left\{\begin{array}{cc}\Gk_\nu, & \Gk_\nu>0,\\
0, &\Gk_\nu\le 0,\\
\end{array}
\right.\quad  
\Gk_\nu^-=\left\{\begin{array}{cc} 0, & \Gk_\nu\ge 0,\\
\Gk_\nu, &\Gk_\nu<0,\\
\end{array}
\right.  
\label{5.3}
\eeq
and
\beq
\Gk^+=\sum_{\nu=0}^N\Gk_\nu^+, \quad \Gk^-=\sum_{\nu=0}^N\Gk_\nu^-.
\label{5.4}
\eeq
For arbitrary complex numbers $C_{\nu j}$, we define
\beq
\GO_0(z)=\sum_{\nu=0}^N\sum_{j=1}^{\Gk_\nu^+}C_{\nu j}M(z,z_{\nu j}).
\label{5.5}
\eeq
Since the function $M(z,\hat z)$ has a single pole $z=\hat z\in\frak F$, 
it is quite clear that the functions $\GO_0(z)+\ov{\GO_0(T(z))}$ and $\GF(z)/\Gc(z)$
have simple poles at the same points of the region $\frak F$.
By the generalized Liouville theorem for multiplicative functions,
\beq
\GF(z)=\Gc(z)[C_0+\GO_0(z)+\ov{\GO_0(T(z))}],
\label{5.6}
\eeq
where $C_0$ is a constant.

For arbitrary constants $C_0$ and $C_{\nu j}$, the function (\ref{5.6})
cannot be accepted as the solution to Problem \ref{p3}. Indeed, it
must be $T$-symmetric, $\frak G$-automorphic and piece-wise meromorphic. This is guaranteed if   
the constant $C_0$ is real and the function $\GF(z)/\Gc(z)$ is multiplicative with factors $H_\nu$.
Because of the quasimultiplicativity property (\ref{4.0})
of the kernel $M(z,\tau)$, however,
in general, 
\beq
\fr{\GF(\Gs_k(z))}{\Gc(\Gs_k(z))}=H_k\fr{\GF(z)}{\Gc(z)}+\Gx_k, \quad k=1,2,\ldots,N,
\label{5.7}
\eeq
where 
\beq
\Gx_k=C_0(1-H_k)+\sum_{\nu=0}^N\sum_{j=1}^{\Gk_\nu^+}[C_{\nu j}\Gz_k(z_{\nu j})
-H_k\ov{C_{\nu j}\Gz_k(z_{\nu j})}],
\label{5.8}
\eeq
where $\Gz_k(z)=M(\Gs_k(z_0),z)$.
Thus, the function $\GF(z)/\Gc(z)$ becomes multiplicative if and only if $\Gx_k=0$.
This condition can be written as follows:
\beq
\I\{H_k^{-1/2}C_0+H_k^{-1/2}\GO_0(\Gs_k(z_0))\}=0.
\quad k=1,2,\ldots,N.
\label{5.9}
\eeq
Next, for negative $\Gk_\nu$, the function (\ref{5.6}) has 
simple poles at the points $z_{\nu j}$. These points become removable 
points if the following conditions are met
\beq
C_0+\GO_0(z_{\nu j})+\ov{\GO_0(T(z_{\nu j}))}=0, \quad j=1,2,\ldots,-\Gk_\nu^-,
\quad \nu=0,1,\ldots,N.  
\label{5.10}
\eeq
Finally, if $-1<\Ga_{\nu 1}\le -\fr12$, then the function (\ref{5.6}) has a nonintegrable
singularity at the point $t_{\nu 1}$. To make   
this singularity integrable we require 
\beq
l_\nu[C_0+\GO_0(t_{\nu 1})+\ov{\GO_0(T(t_{\nu 1}))}]=0, \quad \nu=0,1,\ldots,N,
\label{5.11}
\eeq
where
\beq
l_\nu=\left\{
\begin{array}{cc}
0, & -1/2<\Ga_{\nu 1}\le 0,\\
1, & -1<\Ga_{\nu 1}\le -1/2.\\
\end{array}
\right.  
\label{5.12}
\eeq
The condition (\ref{5.11}) can be simplified. We consider two cases.
If $t_{\nu 1}=t_{01}$, 
then $T(t_{01})=t_{01}$, and 
\beq
C_0+\GO_0(t_{01})+\ov{\GO_0(T(t_{01}))}=C_0+2\R\GO_0(t_{01}).
\label{5.13}
\eeq
If $\nu\ne 0$, then
\beq
\ov{\GO_0(T(t_{\nu 1}))}=\ov{\GO_0(\Gs^{-1}_\nu T_\nu(t_{\nu 1}))}
=H_\nu\ov{\GO_0(t_{\nu 1})}-H_\nu\ov{\GO_0(\Gs_\nu(z_0))}.
\label{5.14}
\eeq
Then, by using the condition (\ref{5.9}), it is easy to verify that
\beq
C_0+\GO_0(t_{\nu 1})+\ov{\GO_0(T(t_{\nu 1}))}  
=H_\nu^{1/2}\R\{H_\nu^{-1/2}[C_0+2\GO_0(t_{\nu 1})-\GO_0(\Gs_\nu(z_0))]\}.
\label{5.15}
\eeq
The two formulas, (\ref{5.13}) and (\ref{5.15}) can be combined, and the conditions
(\ref{5.11})  become
\beq
l_\nu \R\{H_\nu^{-1/2}[C_0+2\GO_0(t_{\nu 1})-\GO_0(\Gs_\nu(z_0))]\}=0,\quad 
\nu=0,1,\ldots,N.
\label{5.16}
\eeq
Thus, there are  $2\Gk^++1$ real  arbitrary constants, $C_0$ and
$C_{\nu j}=C'_{\nu j}+iC''_{\nu j}$. These constants need to satisfy $N$ real conditions (\ref{5.9}),
$-\Gk^-$ complex conditions (\ref{5.10}) and $l=\sum_{\nu=0}^N l_\nu$ real
conditions (\ref{5.16}), in total, $N+l-2\Gk^-$ real conditions.
 
Let $\Gk=\Gk^++\Gk^-=\sum_{\nu=0}^N\Gk_\nu$.
Introduce an integer 
\beq
\frak K=2\Gk-l=\sum_{\nu=0}^N(2\Gk_\nu-l_\nu)
\label{5.17}
\eeq 
and call this number the {\it index} of Problems \ref{p3} and \ref{p2}. Denote by $\Gr$ the rank of the linear system
of $N+l-2\Gk^-$ real equations (\ref{5.9}), (\ref{5.10}), and (\ref{5.16}) 
($1\le\Gr\le N-2\Gk^-+l$).

\begin{theorem}\label{t5.1} {If the index $\frak K$ is negative, then Problem \ref{p3} has only the trivial 
solution. 

If $0\le \frak K\le 2N-2$, then 
Problem \ref{p3} has $2\Gk^+-\Gr+1$ nontrivial solutions (\ref{5.6}) 
over the field of real numbers
provided this number is positive and only the trivial solution otherwise.

If $\frak K>2N-2$, then Problem \ref{p3} has $\frak K-N+1$ solutions (\ref{5.6})
over the field of real numbers.}

\end{theorem}

\begin{proof}
The multiplicative function $\GF(z)/\Gc(z)$ has simple zeros at the points
$z_{\nu j}$ and $T(z_{\nu j})$ if $\Gk_\nu<0$ and at the points $t_{\nu 1}$
if $-1<\Ga_{\nu 1}\le -\fr12$. The number of these points is equal to $l-2\Gk^-$.
The function $\GF(z)/\Gc(z)$ may have some other zeros. That is why, the number of
zeros in the fundamental region is not less than $l-2\Gk^-$. In the case 
$\Gk_\nu>0$, this function has $2\Gk^+$
simple poles at the points $z_{\nu j}$ and $T(z_{\nu j})$.
It is clear that the divisor of the multiplicative function 
$\GF(z)/\Gc(z)$ is a multiple of the divisor
\beq
\frak d=\prod_{\nu=0}^N t_{\nu 1}^{l_\nu}\prod_{j=1}^{|\Gk_\nu|}z_{\nu j}^{-\Gk_\nu}[T(z_{\nu j})]^{-\Gk_\nu},
\label{5.18}
\eeq
and $\deg \frak d=-2\Gk+l=-\frak K$. 

Let first $\frak K<0$. Then the degree of the divisor $\frak d$ is positive and 
since $\frak K=2\Gk^+-(l-2\Gk^-)$, this implies that the number of
zeros of
the multiplicative function $\GF(z)/\Gc(z)$ in the fundamental region $\frak F$
is greater than the number of poles. Such a nontrivial function does not exist.

Let now $\frak K>2N-2$. Notice that the dimension $d^-$ of the space of 
multiplicative forms of weight (-2)
with factors $H_\Gs^{-1}$, $\Gs\in\frak G$, whose divisors are multiples of the divisor 
$\frak d^{-1}$, is equal to zero. Indeed, on one hand $\deg \frak d^{-1}=\frak K$ and 
$\frak K>2N-2$.
On the other hand, the degree of the divisor of any multiplicative form of weight (-2)
on a genus-$N$ Riemann surface is equal to $2N-2$ \cite{kra}.  Let $d^+$ be the dimension of the space  
of multiplicative functions with factors $H_\Gs$, $\Gs\in\frak G$, whose divisors are 
multiples of the divisor $\frak d$. Then by the Riemann-Roch theorem \cite{far},
\beq
d^+=\deg\frak d^{-1}-N+1+d^-,
\label{5.19}
\eeq
and $d^+=\frak K-N+1$. Since the solution $\GF(z)$ has to be a $T$-symmetric
function, Problem 5.1 has $\frak K-N+1$ linearly independent solutions over
the field of real numbers.

In the final case, $0\le \frak K\le 2N-2$, the number of solutions depends on the rank
$\Gr$ of the linear system
of $N+l-2\Gk^-$ real equations (\ref{5.9}), (\ref{5.10}), and (\ref{5.16})
with respect to the $2\Gk^++1$ real unknowns $C_0$,
$C'_{\nu j}$, and $C''_{\nu j}$. 
Let $\tilde\Gr=2\Gk^++1-\Gr$.
Clearly, if $\tilde\Gr\le 0$, then $\GF(z)\equiv 0$.
Otherwise, Problem 5.1 has $\tilde\Gr$ nontrivial solutions defined by (\ref{5.6}). 

\end{proof}

\subsection{Inhomogeneous case: Problem \ref{p2}}  

Introduce the function
\beq
\Psi_0(z)=\fr{1}{4\pi i}\int\limits_LM(z,\tau)\fr{q(\tau)d\tau}{\Gc^+(\tau)}.
\label{5.20}
\eeq
From the Sokhotski-Plemelj formulas and the quasimultiplicativity of the kernel $M(z,\tau)$,
this function has the following properties:
$$
\Psi_0^+(t)-\Psi_0^-(t)=\fr{q(t)}{2\Gc^+(t)},
$$
\beq
\Psi_0(\Gs_\nu(z))=H_\nu\Psi_0(z)+\Psi_0(\Gs_\nu(z_0)).
\label{5.21}
\eeq
By applying the Liouville theorem for multiplicative functions,
we can derive the general solution to Problem \ref{p2} 
\beq
\GF(z)=\Gc(z)[C_0+\GO_0(z)+\ov{\GO_0(T(z))}+\Psi_0(z)+\ov{\Psi_0(T(z))}],
\label{5.22}
\eeq
where $C_0$ is an arbitrary real constant and $\GO_0(z)$ is the function (\ref{5.5}).
As in the homogeneous case, the function $\GF(z)$ has to be invariant 
with respect to the group $\frak G$. This is guaranteed by the following $N$ real conditions
\beq
\I\{H_k^{-1/2}[C_0+\GO_0(\Gs_k(z_0))+\Psi_0(\Gs_k(z_0))]\}=0,
\quad k=1,2,\ldots,N.
\label{5.23}
\eeq
In the case $\Gk_\nu<0$, the function (\ref{5.22}) has 
simple poles at the points $z_{\nu j}$. To remove these poles we require that
\beq
C_0+\GO_0(z_{\nu j})+\Psi_0(z_{\nu j})+\ov{\GO_0(T(z_{\nu j}))}+\ov{\Psi_0(T(z_{\nu j}))}=0.  
\label{5.24}
\eeq
Here $j=1,2,\ldots,-\Gk_\nu^-$, and 
$\nu=0,1,\ldots,N$.  Notice that in the case $-1<\Ga_{\nu 1}\le -\fr12$, in general,
the function $\GF(z)$ has a nonintegrable singularity at the points
$t_{\nu 1}$. The function $\GF(z)$ becomes integrable in this case if
we put
\beq
l_\nu\R\{H_\nu^{-1/2}[C_0+2\GO_0(t_{\nu 1})+2\Psi_0(t_{\nu 1})-\GO_0(\Gs_\nu(z_0))-\Psi(\Gs_\nu(z_0))]\}=0,
\label{5.25}
\eeq
where $\nu=0,1,\ldots,N$, and $l_\nu$ is given by (\ref{5.12}).

Having now written down the linear system of $N+l-2\Gk^-$ real equations (\ref{5.23}) to (\ref{5.25}) for $2\Gk^++1$ real constants, $C_0$, $C'_{\nu j}=\R C_{\nu j}$ and 
$C''_{\nu j}=\I C_{\nu j}$
($j=1,2,\ldots,\Gk_\nu^+$, $\nu=0,1,\ldots,N$),
we can study its solvability. The difference between this system and that in the 
homogeneous case is that now the equations are not homogeneous. 

If $\frak K<0$, then the associated homogeneous system has only a trivial solution.
In the inhomogeneous case, we can exclude all the constants $C_0$, $C'_{\nu j}$,
and $C''_{\nu j}$ from the system (\ref{5.23}) to (\ref{5.25}). This 
leaves us with a new system of  
$N+l-2\Gk^--2\Gk^+-1$ ($\Gk^++\Gk^-=\Gk$)
conditions. If the function $q(t)=\fr{2c(t)}{a(t)-ib(t)}$ in (\ref{2.9}) satisfies these conditions, then the solution to Problem \ref{p2} exists and it is unique.

If $\frak K>2N-2$, then the rank of the system (\ref{5.23}) to (\ref{5.25})
coincides with the number of the equations. Therefore, the system is always solvable
and the general solution to Problem \ref{p2} possesses $2\Gk-l-N+1$ arbitrary real constants.

In the case $0\le \frak K\le 2N-2$ the number of solutions and the number
of solvability conditions depends on the rank $\Gr$ 
($1\le\Gr\le N-2\Gk^-+l$) of 
the system (\ref{5.23}) to (\ref{5.25}). If the solvability conditions
are met, then the general solution may have up to $\frak K-N+1$ arbitrary real constants. 

Thus we have proved the following result. 

\vspace{.1in}
\noindent
\begin{theorem}\label{t5.2} { If the index $\frak K>2N-2$, then Problem \ref{p2}
is always solvable, the general solution possesses $\frak K-N+1$ arbitrary real constants,
and it is given by formula (\ref{5.22}). 

If  $\frak K <0$, then Problem \ref{p2} is solvable 
if and only if the functions $a(t)$, $b(t)$, and $c(t)$
satisfy a system of  $N-\frak K-1$ conditions taken out of $N+l-2\Gk^-$ equations (\ref{5.23}) to (\ref{5.25}). If these conditions are satisfied, then 
the solution (\ref{5.22}) to Problem \ref{p2} is unique.    

If $0\le \frak K\le 2N-2$, then the number of additional conditions from the system
does not exceed 
$N-\frak K-1$. If these conditions are met, then the solution to Problem \ref{p2}
exists, and the number of arbitrary real constants
does not exceed $\frak K-N+1$.}

\end{theorem}

Notice that the theory of solvability to the Hilbert problem (Problem \ref{p1}) 
coincides with that to 
Problem \ref{p2}. The general solution to Problem \ref{p1}
is also given by formula (\ref{5.22}), where we should put $z\in D$.

\setcounter{equation}{0}

\section{The general solution in terms of an automorphic canonical function}

In this section we derive another form of the solution to Problem \ref{p2}.
Instead of the multiplicative canonical function $\Gc(z)$ we  shall use
a piece-wise meromorphic $\frak G$-automorphic canonical function. This function would be 
a particular case of the multiplicative canonical $\Gc(z)$ function (\ref{3.14})
with factors $H_\Gs=1$, $\Gs\in\frak G$ if it did not have extra poles:
\beq
\Gc_a(z)=\exp\{\GG_a(z)+\ov{\GG_a(T(z))}\},
\label{6.1}
\eeq
where 
\beq
\GG_a(z)=\GG(z)+\sum_{\nu=1}^N\left(\int\limits_{r_\nu}^{q_\nu}K(z,\tau)d\tau+\Gl_\nu\int\limits_{L_\nu}
K(z,\tau)d\tau+\Gm_\nu\int\limits_{T(t_{\nu 1})}^{t_{\nu 1}}K(z,\tau)d\tau\right),
\label{6.2}
\eeq
$\GG(z)$ is given by (\ref{3.12}), $r_\nu, q_\nu\in D$, and $\Gl_\nu$ and $\Gm_\nu$ are integers. The points $r_\nu$ are fixed  arbitrarily while $q_\nu$, $\Gl_\nu$, and $\Gm_\nu$ 
are to be determined.
It is assumed that all the points $r_\nu$ and $q_\nu$ ($\nu=1,2,\ldots,N$)
are distinct, and none of them coincides
with the points $z_{\nu j}$ ($j=1,2,\ldots,|\Gk_\nu|$, $\nu=0,1,\ldots,N$). 
It is clear that the function $\Gc_a(z)$ 
possesses the properties (i) to (iv) in Definition \ref{d2} of the multiplicative function $\Gc(z)$.
Similarly to the function $\Gc(z)$, 
\beq
\Gc_a(\Gs_k(z))=\hat H_k^{-1}\Gc_a(z), \quad k=1,2,\ldots,N,
\label{6.3}
\eeq
where
the new factors $\hat H_k$ are given by
\beq
\hat H_k=\exp\{-2i\I\GG_a(\Gs_k(z_*))\}.
\label{6.4} 
\eeq
The function $\Gc_a(z)$ is invariant with respect to the group $\frak G$ if and only if
\beq
\I\GG_a(\Gs_k(z_*))\equiv 0 \quad (\mod \;\pi), \quad k=1,2,\ldots,N,
\label{6.5} 
\eeq
or, equivalently,
\beq
\R\left[\fr{1}{2\pi i}\GG_a(\Gs_k(z_*))\right]\equiv 0 \quad \left(\mod \;\fr12\right), 
\quad k=1,2,\ldots,N.
\label{6.6} 
\eeq
Show next that the conditions (\ref{6.6}) can be considered as the real part
of the classical Jacobi inversion problem for the genus-$N$ Riemann surface $\frak R$.
Let $r_0$ be a fixed point in the domain $D$. Introduce the integrals
\beq
\Gvf_k(z)=\fr{1}{2\pi i}\int\limits_{r_0}^{z}\eta_k(\tau)d\tau, \quad k=1,2,\ldots,N, 
\label{6.7}
\eeq
where $\eta_k(\tau)=K(\Gs_k(z_*),\tau)$. These integrals form the normalized basis
of abelian integrals of the first kind with $A$- and $B$-periods defined
in (\ref{3.4}). By using (\ref{6.7}) and (\ref{3.4}) we can transform the 
conditions (\ref{6.6}) as follows
\beq
\sum_{j=1}^N[\R\Gvf_k(q_j)+\mu_j\R B_{kj}]+\Gl_k\equiv \R d_k \quad \left(\mod\; \fr12\right), \quad k=1,2,\ldots,N,
\label{6.8}
\eeq
where
\beq
d_k=-\fr{1}{2\pi i}\GG(\Gs_k(z_*))+\sum_{j=1}^N\Gvf_k(r_j), \quad k=1,2,\ldots,N.
\label{6.9}
\eeq
Consider now another problem, a modulo-period-1-problem,
\beq
\sum_{j=1}^N[\R\Gvf_k(q_j)+\mu_j\R B_{kj}]\equiv \R d_k \quad (\mod 1), \quad k=1,2,\ldots,N.
\label{6.10}
\eeq 
Evidently, each solution to the system (\ref{6.10}) is a solution to the system
(\ref{6.8}). The new system (\ref{6.10}) can be treated as the "real part"
of the classical Jacobi inversion problem for the surface $\frak R$
\beq
\sum_{j=1}^N[\Gvf_k(q_j)+\mu_jB_{kj}]\equiv \R d_k +i\Ge_k \quad (\mod 1),\quad k=1,2,\ldots,N,
\label{6.11}
\eeq 
where $\Ge_k$ are arbitrary real numbers. It is known \cite{far}
that the solution to this problem, the points 
$q_k$ and the integers $\mu_k$, exist, and its solution can be expressed
through the zeros of the associated genus-$N$ Riemann theta function \cite{zve}, \cite{ant0}. 
Note that the numbers $\Ge _k$ can always be chosen such that
the points $q_k$ coincide with none of the points $r_k$ ($k=1,2,\ldots,N$) 
and $z_{\nu j}$ ($j=1,2,\ldots,|\Gk_\nu|$, $\nu=0,1,\ldots,N$). 

The new canonical function $\Gc_a(z)$, given by (\ref{6.1}) and (\ref{6.2}),
is invariant with respect to the group $\frak G$, $\Gc_a(\Gs(z))=\Gc_a(z)$, $z\in\frak D\setminus\frak L$.
Another difference between this function and the multiplicative function $\Gc(z)$
is the presence of extra zeros and poles of the function $\Gc_a(z)$.
At the points $q_k$ and $T(q_k)$, the function $\Gc_a(z)$ has simple zeros,
and  the points $r_k$ and $T(r_k)$ are simple poles ($k=1,2,\ldots,N$).

We now repeat the procedure of Section 5 adjusting it to the class of 
symmetric piece-wise meromorphic multiplicative functions with factors $H_k=1$. The 
general solution to Problem \ref{p2} has the form
\beq
\GF(z)=\Gc_a(z)[C_0+\GO_a(z)+\Psi_a(z)+\ov{\GO_a(T(z))}+\ov{\Psi_a(T(z))}],
\label{6.12}
\eeq
where
$$
\GO_a(z)=\sum_{\nu=0}^N\sum_{j=1}^{\Gk_\nu^+}C_{\nu j}K(z,z_{\nu j})
+\sum_{j=1}^N A_jK(z,q_j),
$$
\beq
\Psi_a(z)=\fr{1}{4\pi i}\int\limits_L K(z,\tau)\fr{q(t)d\tau}{\Gc_a(\tau)}.
\label{6.13}
\eeq
In comparison to the solution (\ref{5.22}),  the new solution (\ref{6.12}) has
$N$ extra complex arbitrary constants $A_j$, and in total
it has $2\Gk^++2N+1$ real constants. The conditions of solvability 
of the problem consist of $N+l-2\Gk^-$ real equations (\ref{5.23}) to (\ref{5.25}), 
where we should put $H_k=1$ and replace the functions $\GO_0(z)$ and $\Psi_0(z)$
by the functions  $\GO_a(z)$ and $\Psi_a(z)$, respectively.
In addition, to remove the simple poles at $r_j$ and $T(r_j)$
of the automorphic canonical function
$\Gc_a(z)$, we require
\beq
C_0+\GO_a(r_j)+\Psi_a(r_{j})+\ov{\GO_a(r_{j})}+\ov{\Psi_a(T(r_{ j}))}=0, 
\quad j=1,2,\ldots,N.
\label{6.14}
\eeq 
This brings us $2N$ extra real conditions and makes the difference 
between the number of constants and the number of solvability
conditions invariant to the analogue of the Cauchy kernel chosen.

\setcounter{equation}{0}

\section{Piece-wise constant coefficients $a(t)$ and $b(t)$: the solution for the first class
group $\frak G$}

In this section we consider a particular case when the coefficients
$a_\nu(t)$ and $b_\nu(t)$ $t\in L_\nu$  ($\nu=0,1,\ldots,N)$
are piece-wise constant. If, in addition, the group $\frak G$ is a first class group,
the formula for the multiplicative canonical function $\Gc(z)$ can be simplified.
Let
$$
a_\nu(t)=a_{\nu j}=\const, \quad b_\nu(t)=b_{\nu j}=\const,
$$
\beq
\quad t\in t_{\nu\, j}t_{\nu\, j+1}, \quad j=1,2,\ldots,m_\nu,\quad \nu=0,1,\ldots,N,
\quad t_{\nu\,m_\nu+1}=t_{\nu 1}.
\label{7.1} 
\eeq
In this case $p(\tau)$
is also a piece-wise constant function,
\beq
p(\tau)=p_{\nu j}=-\fr{a_{\nu j}+ib_{\nu j}}{a_{\nu j}-ib_{\nu j}}, 
\quad \tau\in t_{\nu\, j}t_{\nu\, j+1}, \quad j=1,2,\ldots,m_\nu,\quad \nu=0,1,\ldots,N.
\label{7.2}
\eeq
According to the inequalities (\ref{3.8}), (\ref{3.10}), and (\ref{3.11})
the values of the piece-wise function $\arg p(\tau)=\arg p_{\nu j}$ 
and the integers $\Gk_\nu$ are defined by
$$
-\pi<\arg p_{\nu 1}\le \pi,
$$$$
-2\pi<\arg p_{\nu\,j-1}-\arg p_{\nu\, j}\le 0, \quad j=2,\ldots, m_\nu, \quad
\nu=0,1,\ldots,N,
$$
\beq
-4\pi<\arg p_{\nu m_{\nu}}-\arg p_{\nu 1}-4\pi\Gk_\nu\le 0, \quad \nu=0,1,\ldots,N.
\label{7.3}
\eeq
Assuming that $\frak G$ is a first class group, evaluate the integrals in (\ref{3.12}).
In this case, the kernel $K(z,\tau)$ is a uniformly and absolutely convergent series
(\ref{3.6}), and formula (\ref{3.12}) reads
\beq
\GG(z)=\fr{1}{4\pi}\sum_{\nu=0}^N\sum_{j=1}^{m_\nu}
\arg p_{\nu j}
\int\limits_{t_{\nu j}}^{t_{\nu\, j+1}}\sum_{\Gs\in\frak G}
\fr{\Gs'(\tau)}{\Gs(\tau)-z}d\tau +\sum_{\nu=0}^N\sgn\Gk_\nu\sum_{j=1}^{|\Gk_\nu|}\int\limits_{t_{\nu 1}}
^{z_{\nu j}}\fr{\Gs'(\tau)}{\Gs(\tau)-z}d\tau.  
\label{7.4}
\eeq
Evaluating the integrals and using formulas (\ref{3.16}) we can write
\beq
\GG(z)=\ln\prod_{\Gs\in\frak G}\prod_{\nu=0}^N\prod_{j=1}^{m_\nu}
(\Gs(t_{\nu j})-z)^{\Ga_{\nu j}}
\prod_{j=1}^{|\Gk_\nu|}(\Gs(z_{\nu j})-z)^{\sgn\Gk_\nu}.
\label{7.5}
\eeq
For the canonical function $\Gc(z)$, we also need $\GG(T(z))$. Since $t_{\nu j}\in L_\nu$,
we have $t_{\nu j}=T(t_{\nu j})$. Making the substitution $T\Gs T_\nu=\Go\in\frak G$ and
$T\Gs T=\Go_*\in\frak G$
we can establish the following relations:
$$
\ov{\Gs(t_{\nu j})-T(z)}=\fr{\Gr_0^2(z-\Go(t_{\nu j}))}{(\Go(t_{\nu j})-\Gd_0)(z-\Gd_0)},
$$
\beq
\ov{\Gs(z_{\nu j})-T(z)}=
\fr{\Gr_0^2(z-\Go_*(z^*_{\nu j}))}{(\Go_*(z^*_{\nu j})-\Gd_0)(z-\Gd_0)}.
\label{7.6}
\eeq
Here $z^*_{\nu j}=T(z_{\nu j})$.
Combining the two equalities in (\ref{7.6}) with the expression for $\ov{\GG(T(z))}$
obtained from (\ref{7.5}) we derive the canonical function (\ref{3.14})
\beq
\Gc(z)=\left(\fr{\Gr_0^2}{z-\Gd_0}\right)^\Gg\GP(z),
\label{7.7}
\eeq
where
$$
\Gg=\sum_{\nu=0}^N\left(\sum_{j=1}^{m_\nu}\Ga_{\nu j}+\Gk_\nu\right),
$$
\beq
\GP(z)=\prod_{\Gs\in\frak G}\prod_{\nu=0}^N\left[
\prod_{j=1}^{m_\nu}\left(\fr{(\Gs(t_{\nu j})-z)^2}{\Gd_0-\Gs(t_{\nu j})}
\right)^{\Ga_{\nu j}}
\prod_{j=1}^{|\Gk_{\nu}|}\left(\fr{(z-\Gs(z_{\nu j}))(z-\Gs(z^*_{\nu j}))}{\Gd_0-\Gs(z^*_{\nu j})}\right)^{\sgn\Gk_\nu}
\right].
\label{7.8}
\eeq
This formula can further be simplified. Indeed, from the definition (\ref{3.16})
of the numbers $\Ga_{\nu j}$ and from the first formula in (\ref{7.8})
we derive $\Gg=0$, and therefore $\Gc(z)=\GP(z).$ 

\begin{example}\label{e7.1} Consider a particular case of Problem 2.1 when  all
$m_\nu$ are even: $m_\nu=2n_\nu$ ($\nu=0,1,\ldots,N$), and  
$$
\R\Gf(t)=c(t), \quad t\in t_{\nu j}t_{\nu j+1}, \quad j=1,3,\ldots, 2n_\nu-1,
$$
\beq
\I\Gf(t)=c(t), \quad t\in t_{\nu j}t_{\nu j+1}, \quad j=2,4,\ldots, 2n_\nu,
\label{7.9}
\eeq
and $c(t)$ is continuous on $t_{\nu j} t_{\nu j+1}$, $j=1,2,\ldots, 2n_\nu$.
In this case,
\beq 
p_{\nu j}=\left\{\begin{array}{cc}
-1, & j=1,3,\ldots, 2n_\nu-1,\\
1, & j=2,4,\ldots, 2n_\nu.\\
\end{array}
\right.
\label{7.10}
\eeq
From  the definition (\ref{7.3}) of $\arg p_{\nu j}$ and the numbers $\Gk_\nu$,
\beq
\arg p_{\nu j}=\pi j, \quad j=1,2,\ldots, 2n_\nu,
\quad \Gk_\nu=\left[\fr{n_\nu+1}{2}\right].
\label{7.11}
\eeq
where $[a]$ is the integer part of a number $a$.
This implies
\beq
\Ga_{\nu 1}=\left\{\begin{array}{cc}
-3/4, & n_\nu=2s_\nu-1,\\
-1/4, & n_\nu=2s_\nu,\\
\end{array}\right.
\quad 
\Ga_{\nu j}=-\fr14, \quad j=2,3,\ldots,2n_\nu,\quad 
\nu=0,1,\ldots,N.
\label{7.12}
\eeq
We now observe that $\Gk^-_\nu=0$, $\Gk^+_\nu=\Gk_\nu=[(n_\nu+1)/2]$.

Simple computations show that in both the cases, $\Ga_{\nu 1}=-\fr34$ and $\Ga_{\nu 1}=-\fr14$,
the index of the problem is $\frak K=\sum_{\nu=0}^Nn_\nu$. The general solution (\ref{5.22})
possesses  $\frak K+l+1$ arbitrary constants $C_0$ and $C_{\nu j}$, $\nu=0,1,\ldots, N$,
$j=1,2,\ldots,\Gk_\nu$. The solution has to satisfy $N+l$ conditions of solvability
(\ref{5.23}) and (\ref{5.25}).                
The difference between the number of the arbitrary constants
and the number of the conditions is $\frak K+1-N$. 
\end{example}

\setcounter{equation}{0}

\section{Circular $(N+1)$-connected Hall plate with electrodes and dielectrics}

\subsection{Statement of the problem}

Consider a semiconductor $D$, an infinite $N+1$-connected circular plate with finite
contacts on the circles  
$L_\nu$ which form the boundary of the structure (Fig. 2). We assume that on the circles
$L_\nu$ ($\nu =0,1,\ldots,N_0$), the number of the electrodes is even, $n_\nu=2s_\nu$,
and
on the other circles $L_\nu$ ($\nu =N_0+1,N_0+2,\ldots,N$), the number of electrodes is odd,
$n_\nu=2s_\nu-1$.
Here $N_0\in[-1,N]$. If $N_0=-1$, then all the circles have an odd number of electrodes.
If $N_0=N$, then all the circles have an even number of electrodes. 
Let the $j$-th electrode on the circle $L_\nu$ be denoted as 
$e_{\nu j}=t_{\nu\,2j-1}t_{\nu\,2j}$,
$j=1,2,\ldots,n_\nu$. It is assumed that the rest of the boundary of each circle is insulated.
\begin{figure}[t]
\centerline{
\scalebox{0.6}{\includegraphics{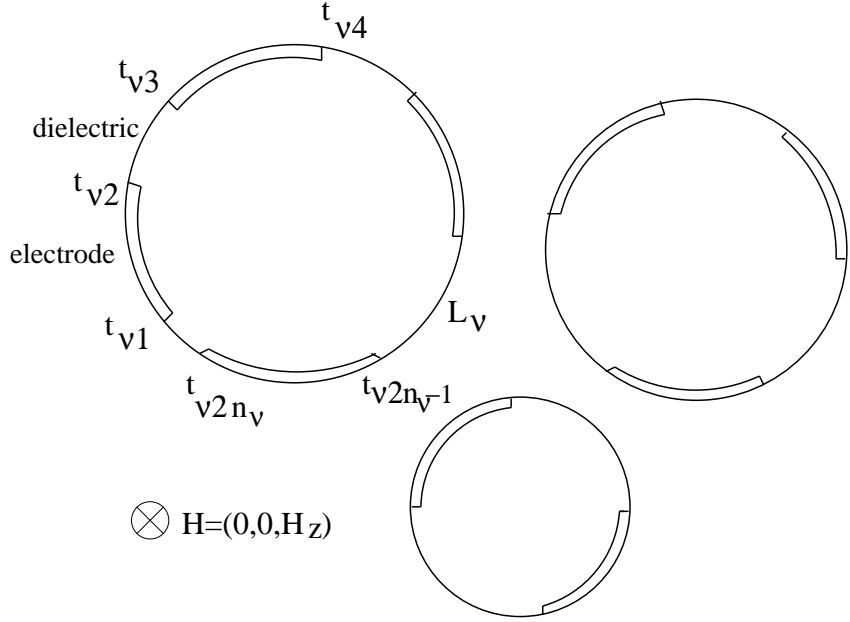}}
}
\caption{An infinite multiply connected Hall plate.}
\label{fig2}
\end{figure}

Let the magnetic field be orthogonal to the plate, and its intensity $\BH$ be prescribed,
$\BH=(0,0,H_z)$, $H_z=\const$. 
The system is activated by applied
electric field flowing 
through the electrodes  
\beq
J_{\nu j}=h_0\int\limits_{e_{\nu j}}J_n d\tau,\quad j=1,2,\ldots,n_\nu,
\quad \nu=0,1,\ldots,N.
\label{8.1}
\eeq
Here  $h_0$ is the thickness of the plate,
$J_n$ is the normal component of the current density $\BJ=(J_x,J_y,0)$,
and $J_{\nu j}$ are the total currents
flowing through the electrodes. Assume also that at infinity
there is no source of an external current. Then 
the electric field intensity $\BE=(E_x,E_y,0)$ vanishes at infinity,
\beq
\BE=\fr{\BA_0}{z}+O(z^{-2}), \quad z\to\infty,
\label{8.2'}
\eeq
where $\BA=(A_x,A_y,0)$ is a constant nonzero vector, and the currents
$J_{\nu j}$ have to be prescribed such that
\beq
\sum_{\nu=0}^N\sum_{j=1}^{n_\nu}J_{\nu j}=0.
\label{8.2}
\eeq
Because of the applied electric and magnetic fields, the semiconductor
develops a component of electric field orthogonal
to both the electric and magnetic fields. This phenomenon, known as the Hall effect,
is described by the generalized Ohm's law
\beq
\BE=\Ga\BJ-R_H\BJ\wedge\BH,
\label{8.3}
\eeq
where 
$\Ga$ is the resistivity in the absence of the magnetic field, and $R_H$ is the Hall 
coefficient.
The Maxwell equations written for a source-free 2-d-medium in the steady-state case 
imply the harmonicity of the current in the domain $D$. On the electrodes,
the tangential component $E_\tau$ of the electric field intensity vanishes, whilst
on the dielectrics (insulated walls), the normal component $J_n$
of the current intensity vanishes:
$$
-E_x\sin\Gt+E_y\cos\Gt=0, \quad t\in e_{\nu},
$$
\beq
J_x\cos\Gt+J_y\sin\Gt=0, \quad t\in d_{\nu},
\label{8.4}
\eeq
where $\Gt$ is the polar angle in the parametrization of the circle $L_\nu$,
$t-\Gd_\nu=\Gr_\nu e^{i\Gt}$, $e_{\nu}$ and $d_\nu$ are the unions of the electrodes
are the dielectrics on the circle $L_\nu$, respectively,
\beq
e_{\nu}=\bigcup_{j=1,3,\ldots}^{2n_\nu-1}t_{\nu j}t_{\nu j+1} ,\quad
d_{\nu}=\bigcup_{j=2,4,..}^{2n_\nu} t_{\nu j}t_{\nu j+1},
\label{8.5}
\eeq
and $t_{\nu 2n_\nu+1}=t_{\nu 1}$. By employing the Ohm's law (\ref{8.3}) rewrite the first
boundary condition in (\ref{8.4}) in the form
\beq
(\Ga\sin\Gt+R_H\cos\Gt)J_x-(\Ga\cos\Gt-R_H\sin\Gt)J_y=0.
\label{8.6}
\eeq
Introduce next a new function, $\Gf(z)=J_x-iJ_y$, analytic in the domain $D$
and satisfying the Hilbert boundary condition
\beq
a(t)u(t)+b(t)v(t)=0,
\label{8.6'}
\eeq
where
$$
u(t)=J_x,\quad v(t)=-J_y,
$$
$$
a(t)=\left\{\begin{array}{cc}
(\Ga+i\Gb)(t-\Gd_\nu)^2-(\Ga-i\Gb)\Gr_\nu^2, & t\in e_{\nu},\\
(t-\Gd_\nu)\Gr_\nu^{-1}+\Gr_\nu(t-\Gd_\nu)^{-1}, & t\in d_{\nu},\\
\end{array}\right.
$$
\beq
b(t)=\left\{\begin{array}{cc}
i[(\Ga+i\Gb)(t-\Gd_\nu)^2+(\Ga-i\Gb)\Gr_\nu^2], & t\in e_{\nu },\\
i[(t-\Gd_\nu)\Gr_\nu^{-1}-\Gr_\nu(t-\Gd_\nu)^{-1}], & t\in d_{\nu },\\
\end{array}\right.
\label{8.6''}
\eeq
$\Gb=H_zR_H$. Clearly, (\ref{8.6'}) is a particular case ($c(t)=0$) of the boundary condition (\ref{2.2}).
The function $\Gf(z)$ is sought in the class of functions which are holomorphic in $D$, $H$-continuous in 
$D\cup L$ except for the points of the set $\GT=\cup_{\nu=0}^N \GT_{\nu}$,
$\GT_\nu=\{t_{\nu 1}, \ldots, t_{\nu 2n_\nu}\}$, where it may have integrable singularities.

\subsection{Solution to the Hilbert problem}

As it was shown in Section 2 for the general case, the Hilbert problem (\ref{8.6'})
is equivalent to the homogeneous Riemann-Hilbert problem
\beq
\GF^+(t)=p(t)\GF^-(t), \quad t\in L\setminus\GT.
\label{8.6.0}
\eeq
for the function $\GF(z)$
defined in (\ref{2.3}). The coefficient in (\ref{8.6.0}) is
$p(t)=p_1(t)p_2(t)$,
where $p_1(t)$ is a continuous function and $p_2(t)$ is a piecewise constant 
function given by
$$
p_1(t)=-\fr{\Gr_\nu^2}{(t-\Gd_\nu)^2}, \quad t\in L_\nu, 
$$
$$
p_2(t)=p_{\nu j}, \quad t\in t_{\nu j}t_{\nu j+1}, \quad j=1,2,\ldots,2n_\nu,
$$ 
\beq
p_{\nu j}=\left\{\begin{array}{cc}
-(\Ga-i\Gb)(\Ga+i\Gb)^{-1}, & j=1,3,\ldots,2n_\nu-1,\\
1, & j=2,4,\ldots,2n_\nu.\\
\end{array}\right.
\label{8.7}
\eeq
In addition, because of the condition (\ref{8.2'}), it is required that the solution has the following 
asymptotics at infinity
\beq
\GF(z)=\fr{K}{z}+O(z^{-2}), \quad z\to\infty, 
\label{8.7'}
\eeq
where $K$ is a nonzero constant.

We split the canonical function of the problem as follows:
\beq
\Gc(z)=\Gc_1(z)\Gc_2(z),
\label{8.8}
\eeq
where the first function, $\Gc_1(z)$, factorizes the continuous function $p_1(t)$, and
the second one factorizes the piece-wise constant function $p_2(t)$.
In order to find the function $\Gc_1(z)$, choose the starting point $t_{\nu 1}\in L_\nu$.
Then
the function $\Gc_1(z)$ is determined by
\beq
\Gc_1(z)=\Gc_*(z)\exp\left\{\GG_1(z)+\ov{\GG_1(T(z))}\right\},
\label{8.9}
\eeq   
where $\Gc_*(z)$ is a piece-wise automorphic function in the domain $\frak D\setminus\frak L$ such that
\beq
\Gc_*(z)=\left\{\begin{array}{cc}
i,&\; z\in \Gs(D),\\
-i,&\; z\in \Gs(T(D)),\\
\end{array}
\right.\quad \Gs\in\frak G.
\label{8.10}  
\eeq
This function satisfies the boundary condition $\Gc_*(t)=-\Gc_*(t)$,
$t\in L_\nu$ ($\nu=0,1,\ldots,N$), and the symmetry and automorphicity conditions (\ref{2.5}) and (\ref{2.4}).
The function $\GG_1(z)$
is determined by the singular integrals
\beq
\GG_1(z)=\sum_{\nu=0}^N\fr{1}{2\pi i}\int\limits_{L_\nu}\ln\fr{\Gr_\nu}{\tau-\Gd_{\nu }}K(z,\tau)d\tau,
\label{8.11}
\eeq
where 
a branch of the logarithmic function $\ln[\Gr_\nu(z-\Gd_\nu)^{-1}]$  is fixed in the $z$-plane cut along a line
joining the branch points $z=\Gd_{\nu}$ and $z=\infty$
and passing through the point $t_{\nu 1}$.  
The function $\Gc_1(z)$ is a piece-wise $\frak G$-multiplicative  function,
\beq
\Gc_1(\Gs_j(z))=[H^{(1)}_{j}]^{-1}\Gc_1(z),\quad j=1,2,\dots,N,
\label{8.12}
\eeq
where
\beq
H^{(1)}_{j}=\exp\{-2i\I h_j^{(1)}\}, \quad 
h_j^{(1)}=\sum_{\nu=0}^N\fr{1}{2\pi i}\int\limits_{L_\nu}\ln\fr{\Gr_\nu}{\tau-\Gd_\nu}\eta_j(\tau)d\tau.
\label{8.13}
\eeq
For the first class groups, the integrals in the expression for the function $\Gc_1(z)$ can be evaluated. By 
choosing
$z_*=\infty$ we obtain 
\cite{ant1}
\beq
\Gc_1(z)=A_1\fr{\Gc_*(z)q(z)}{z-\Gd_0},\quad z\in D\cup T(D),
\label{8.14}
\eeq
where
$$
q(z)=(z-t_{0 1})
\prod_{\nu=1}^N\fr{z-t_{\nu 1}}{z-\Gd_\nu}\prod_{j=0}^N\left(
\prod_{\Gs\in\frak G_j'}\fr{z-\Gs(t_{j 1})}{z-\Gs(\infty)} 
\prod_{\Gs\in\frak G_j''}\fr{z-\Gs(t_{j 1})}{z-\Gs(\Gd_j)}\right), 
$$
\beq
A_1=\fr{\Gr_0}{\sqrt{q(\Gd_0)}},
\label{8.15}
\eeq
where  $\frak G'_j$ is the set of all transformations $T_{m_{2\mu}}T_{m_{2\mu-1}}\ldots
T_{m_2}T_j$, $m_2\ne j$, $m_3\ne m_2$, $\ldots$, $m_{2\mu}\ne m_{2\mu-1}$, $\mu=1,2,\ldots$,
where all the indices $m_\Gm$ vary from $0$ to $n$: $m_{\Gm}=0,1,\ldots,n$,
The set $\frak G''_j=\frak G\setminus\frak G'_j\setminus\Gs_0$ includes all the other transformations 
$T_{m_{2\mu}}T_{m_{2\mu-1}}\ldots T_{m_2}T_{m_1}$ ($m_1\ne j$) of the group $\frak G$ except for
the identical transformation $\Gs_0$. 

Notice that the constant $A_1$ in the representation (\ref{8.14}) of the canonical function
$\Gc_1(z)$ cannot be removed. It is needed to satisfy the symmetry condition
$\Gc_1(z)=\ov{\Gc_1(T(z))}$. The branch in (\ref{8.15}) is chosen arbitrarily.
Its choice affects the sign of the constant and does not break the symmetry of the canonical
function. 
We shall also need a series representation of the coefficients $h_\nu^{(1)}$. For the first class
group $\frak G$, formula (\ref{8.13}) yields \cite{ant1}
$$
2h_\nu^{(1)}=\ln\fr{T_\nu(\Gd_0)-\Gd_\nu}{\Gr_\nu}+\sum_{j=0}^N\left[
\ln\fr{T_\nu(\Gd_0)-t_{j1}}{T_\nu(\Gd_0)-\Gd_j}
\right.
$$
\beq
\left.
+\sum_{\Gs\in\frak G_j'}
\ln\fr{T_\nu(\Gd_0)-\Gs(t_{j1})}{T_\nu(\Gd_0)-\Gs(\infty)}+\sum_{\Gs\in\frak G_j''}
\ln\fr{T_\nu(\Gd_0)-\Gs(t_{j1})}{T_\nu(\Gd_0)-\Gs(\Gd_j)}\right].
\label{8.16}
\eeq
In Section 7 we have determined the canonical function (\ref{7.7}) and showed that $\Gg=0$.
To use this formula for the function $\Gc_2(z)$, we specify the parameters.
Notice that the numbers $m_\nu$ are always even, $m_\nu=2n_\nu$, and
the parameter $\Gb$ can be any finite real number.
Introduce the parameter
\beq
\Gd=2\tan^{-1}\fr{\Ga}{|\Gb|}\in(0,\pi).
\label{8.17}
\eeq
Let first $\Gb>0$.
According to the inequalities (\ref{7.3})
we choose $\arg p_{\nu j}$ as follows
$$
\arg p_{\nu j}=\Gd+(j-1)\pi, \quad j=1,3,\ldots,2n_\nu-1,
$$ 
\beq
\arg p_{\nu j}=j\pi, \quad j=2,4,\ldots,2n_\nu.
\label{8.18}
\eeq
It remains now to write down the integers $\Gk_\nu$ and the parameters $\Ga_{\nu j}$. Since the function
$\Gc_1(z)$ has simple zeros at the points $t_{\nu 1}$ it will be convenient to 
choose the integers $\Gk_\nu$ such that the following inequalities hold:
\beq
-4\pi<\arg p_{\nu \, 2n_\nu}-\arg p_{\nu 1}-4\pi\Gk_\nu+2\pi\le 0,
\label{8.19}
\eeq
from which we obtain
\beq
\Gk_\nu=\left[\fr{n_\nu}{2}\right]+1=\left\{\begin{array}{cc}
s_\nu+1, & n_\nu=2s_\nu,\\
s_\nu, & n_\nu=2s_\nu-1.\\
\end{array}\right.
\label{8.20}
\eeq
We also need to determine the constants $\Ga_{\nu j}$,
$$
\Ga_{\nu 1}=\fr{1}{4\pi}(\arg p_{\nu\, 2n_\nu}-\arg p_{\nu 1}+2\pi)-\Gk_\nu=
\left\{\begin{array}{cc}
-\Gd_*-\fr12, & n_\nu=2s_\nu,\\
-\Gd_*, & n_\nu=2s_\nu-1,\\
\end{array}\right.
$$
\beq 
\Ga_{\nu j}=\fr{1}{4\pi}(\arg p_{\nu\, j-1}-\arg p_{\nu\, j})=\left\{\begin{array}{cc}
-\Gd_*, & j=3,5,\ldots,2n_\nu-1,\\
\Gd_*-\fr12, & j=2,4,\ldots,2n_\nu,\\
\end{array}\right.
\label{8.20'}
\eeq
where
\beq
\Gd_*=\fr{\Gd}{4\pi}, \quad \nu=0,1,\ldots,N.
\label{8.21}
\eeq
Consider next the case $\Gb<0$. The parameters of interest are defined as follows:
$$
\arg p_{\nu j}=-\Gd+(j-1)\pi, \quad j=1,3,\ldots,2n_\nu-1,
$$ 
$$
\arg p_{\nu j}=(j-2)\pi, \quad j=2,4,\ldots,2n_\nu,
$$
$$
\Ga_{\nu 1}=\left\{\begin{array}{cc}
\Gd_*-1, & n_\nu=2s_\nu-1,\\
\Gd_*-\fr12, & n_\nu=2s_\nu,\\
\end{array}\right.
$$
\beq 
\Ga_{\nu j}=\left\{\begin{array}{cc}
\Gd_*-\fr12, & j=3,5,\ldots,2n_\nu-1,\\
-\Gd_*, & j=2,4,\ldots,2n_\nu,\\
\end{array}\right.
\quad 
\nu=0,1,\ldots,N,
\label{8.22}
\eeq
and the integers $\Gk_\nu=[n_\nu/2]+1$ are the same as in the case $\Gb>0$. 
Now, with the parameters $\Ga_{\nu j}$ and $\Gk_\nu$ being defined we can write down the
function $\Gc_2(\Gz)$
\beq
\Gc_2(z)=\prod_{\Gs\in\frak G}\prod_{\nu=0}^N\left[
\prod_{j=1}^{2n_\nu}\left(\fr{(\Gs(t_{\nu j})-z)^2}{\Gd_0-\Gs(t_{\nu j})}
\right)^{\Ga_{\nu j}}
\prod_{j=1}^{\Gk_{\nu}}\left(\fr{(z-\Gs(z_{\nu j}))(z-\Gs(z^*_{\nu j}))}{\Gd_0-\Gs(z^*_{\nu j})}\right)
\right].
\label{8.23}
\eeq
The function $\Gc_2(z)$ is the multiplicative canonical 
function with the factors $H^{(2)}_{j}$ obtained 
by replacing the factors $H_j$ derived  in Section 3 by $H^{(2)}_{j}=\exp(-2i\I h_\nu^{(2)})$, $h_\nu^{(2)}=\GG_2(\Gs_\nu(z_*))$. 
The function $\GG_2(z)$ coincides with the function $\GG(z)$ in (\ref{3.12})
if $p(\tau)$
is replaced by $p_2(\tau)$.
 
Having found the functions $\Gc_1(z)$ and $\Gc_2(z)$, we write down the general solution of the Riemann-Hilbert problem
\beq
\GF(z)=\Gc_1(z)\Gc_2(z)[C_0+\GO_0(z)+\ov{\GO_0(T(z))}],\quad 
\label{8.24}
\eeq
where 
\beq
\GO_0(z)=\sum_{\nu=0}^N\sum_{j=1}^{\Gk_\nu}C_{\nu j}M(z,z_{\nu j}),
\label{8.25}
\eeq
$M(z,\tau)$ is the quasimultiplicative kernel with the factors $H_j=H_j^{(1)}H_j^{(2)}$,
$C_0$ is a real constant and $C_{\nu j}=C_{\nu j}'+iC_{\nu j}''$ are complex constants.

\subsection{Definition of the constants}

In total, the general solution (\ref{8.24}) possesses $2\Gk+1$ real constants to be determined.
It will be convenient to have another representation of the number of the constants.
Since
\beq
\Gk=\sum_{\nu=0}^N\Gk_\nu,\quad \Gk_\nu=\left[\fr{n_\nu}{2}\right]+1,
\label{8.26}
\eeq
$n_\nu=2s_\nu$ if $\nu=0,1,\ldots, N_0$, and $n_\nu=2s_\nu-1$ if $\nu=N_0+1,N_0+2,\ldots, N$, we obtain
\beq
2\Gk+1=2s+2N_0+3, \quad s=\sum_{\nu=0}^N s_\nu.
\label{8.27}
\eeq
Show next that for the definition of these constants, we have the same number of 
linear equations. The first $N$ equations come from the conditions which guarantee that 
the solution $\GF(z)$ is invariant with respect to the group $\frak G$ 
\beq
\I\{H_k^{-1/2}[C_0+\GO_0(T_k(\Gd_0))]\}=0,
\quad k=1,2,\ldots,N.
\label{8.28}
\eeq
Because the zero of the kernel $M(z,\tau)$, the point $z_0\in\frak F$, can be chosen arbitrary we take here
and further $z_0=\infty$. This gives us $\Gs_k(\infty)=T_k(\Gd_0)$ and simplifies
the additional conditions.

The function $\GF(z)$ has to have a simple zero at infinity. This is guaranteed
by the complex condition 
$C_0+\ov{\GO_0(\Gd_0)}=0$, which is equivalent to the following two real equations: 
$$
C_0=-\R\{\GO_0(\Gd_0)\}, 
$$
\beq
\I\{\GO_0(\Gd_0)\}=0. 
\label{8.29}
\eeq
We have shown that if $n_\nu=2s_\nu$, then the parameter $\Ga_{\nu 1}=-\Gd_*-\fr12$ when $\Gb>0$
and $\Ga_{\nu 1}=\Gd_*-1$ when $\Gb<0$. This means that $2\Ga_{\nu 1}\in(-2,-1)$ when the number
of electrodes is even. Otherwise, $2\Ga_{\nu 1}\in(-1,0)$.  Thus, if $n_\nu=2s_\nu$, then the solution 
(\ref{8.24}) has a nonintegrable singularity at the points $t_{\nu 1}$. It becomes
an integrable singularity if the conditions (\ref{5.16}) hold. As in equations (\ref{8.28}) and (\ref{8.29}), we put
$z_0=\infty$ and use the fact that $l_\nu=1$ if $\nu=0,1,\ldots, N_0$ and $l_\nu=0$ for
$\nu>N_0$. This transforms equations (\ref{5.16}) to the following $N_0+1$ conditions
\beq
\R\{H_\nu^{-1/2}[C_0+2\GO_0(t_{\nu 1})-\GO_0(T_\nu(\Gd_0))]\}=0,\quad 
\nu=0,1,\ldots,N_0.
\label{8.30}
\eeq
In addition, we have the physical conditions (\ref{8.1})
which can be written in the form
\beq
\int\limits_{e_{\nu j}}J_n d\tau=\fr{J_{\nu j}}{h_0},\quad j=1,2,\ldots,n_\nu,
\quad \nu=0,1,\ldots,N,
\label{8.31}
\eeq
where
\beq
J_n=J_x\cos\Gt+J_y\sin\Gt, \quad J_x=\R\GF(z), \quad J_y=-\I\GF(z).
\label{8.32}
\eeq
The number $n=n_0+\ldots+n_N$ of the conditions (\ref{8.31}) can be expressed through the integers
$s$, $N$, and $N_0$ as follows:
\beq
n=2s-N+N_0.
\label{8.33}
\eeq
Thus, we obtain a system which consists of $N$ equations (\ref{8.28}), two conditions (\ref{8.29}),
$N_0+1$ relations (\ref{8.30}) and $2s-N+N_0$ equations (\ref{8.31}). As we had anticipated,
in total, there are $2s+2N_0+3$ real equations for the determination of $2s+2N_0+3$ real constants.
 
The first equation in (\ref{8.29}) expresses the real constant $C_0$ through the constants $C_{\nu j}$.
This makes possible to simplify the solution of the problem. The final formula has $2\Gk=2s+2N_0+2$
real constants and it becomes
\beq
J_x-iJ_y=\sum_{\nu=0}^N\sum_{j=1}^{\Gk_\nu}[C_{\nu j}'S_{\nu j}^+(z)+iC_{\nu j}''S_{\nu j}^-(z)],\quad z\in D,
\label{8.34}  
\eeq
where the functions $S^\pm_{\nu j}$ are free of the constants $C_{\nu j}$,
\beq
S_{\nu j}^\pm(z)=\Gc_1(z)\Gc_2(z)\left[-\fr12M(\Gd_0,z_{\nu j})\mp\fr12\ov{M(\Gd_0,z_{\nu j})}
+M(z,z_{\nu j})\pm\ov{M(T(z),z_{\nu j})}\right].
\label{8.35}
\eeq


\section*{Conclusions}

We have developed a method for the Hilbert problem for a circular multiply
connected domain and the Riemann-Hilbert problem for piece-wise analytic functions
invariant with respect to a symmetric Schottky group. The coefficients of both the problems
are piece-wise H\"older continuous functions, and the discontinuities of the coefficients
cause integrable singularities of the solution. The technique we have proposed requires
the use of two analogues of the Cauchy kernel, a quasiautomorphic kernel and
a quasimultiplicative kernel. We have proved the existence results for both the kernels. 
The existence of the former kernel follows from the theory of abelian integrals on a compact 
Riemann surface. To prove the existence of the quasimultiplicative kernel, we have
used the Riemann-Roch theorem for multiplicative functions. 
For the first class groups (the Burnside classification), the solution to the 
Hilbert and the Riemann-Hilbert problems have been derived in a series form. In addition,
we have obtained the solution in terms of an automorphic analogue of the Cauchy kernel.
It turns out that the use of this kernel requires the solution of the Jacobi inversion 
problem. We emphasize that the procedure which is based on the quasiautomorphic and
quasimultiplicative kernels bypasses the Jacobi inversion problem. There is another
advantage to employ the quasiautomorphic and quasimultiplicative  kernels, not the
automorphic kernel. For a ($N+1$)-connected circular domain
the second method leads to a solution which has $2N$ extra real constants and therefore, in comparison 
with the first method, there are $2N$ extra equations to be solved.  

The method proposed has been illustrated by the solution of a model electromagnetic
steady-state problem on the motion of charged electrons in a plate when the applied magnetic field
is orthogonal to the plate. The plate, known as a Hall plate, has $N+1$ circular holes
with electrodes and dielectrics on the walls. We have reduced the problem to a particular
case of the Hilbert problem with the coefficient $p(t)=p_1(t)p_2(t)$. The first function is 
continuous, and its factorization has been implemented by the method \cite{ant1}. The second function, $p_2(t)$,
is a piece-wise constant function. Because of this property and also because the first canonical function
has a zero at the starting point, we have managed to simplify the general formula for the second canonical function.
We have derived the exact formula for the current density. The formula possesses a finite number, $2\Gk$,
of unknown constants which solve a system of $2\Gk$ linear algebraic equations,
where 
$2\Gk=2s+2N_0+2$, $N_0$ is the number of circles with an even number of electrodes,
$s=s_0+s_1+\ldots s_N$, $s_\nu=[(n_\nu+1)/2]$, and $n_\nu$ is  the number of electrodes on the $\nu$-th circle. Note that the system of equations
for the constants consists of $N+N_0+2$ "mathematical" equations (due to the method) and $2s+N_0-N\ge 1$ physical equations. 

Finally, we notice that the technique proposed can be extended for polygonal multiply connected domains.
This can be done by implementing a two-step-procedure. First, one needs to map an $(N+1)$-connected circular
domain into an ($N+1$)-connected polygonal domain \cite{del}, \cite{cro} and define
the coefficients $a(t)$, $b(t)$, and $c(t)$. The second step is to use 
the solution to the Hilbert problem for ($N+1$)-connected circular domain derived in this paper.

\bibliographystyle{amsplain}

\end{document}